\documentclass[12pt]{article}
\usepackage{fourier}
\usepackage[utf8]{inputenc}
\usepackage[T1]{fontenc}

\usepackage[a4paper,margin=1in,footskip=0.25in]{geometry}

\usepackage{titling}
\usepackage{amssymb}
\usepackage{amsmath}
\usepackage{amsthm}
\usepackage{tikz-cd}
\usepackage{blindtext}
\usepackage{hyperref}
\hypersetup{
colorlinks,
linkcolor={red!50!black},
citecolor={blue!50!black},
urlcolor={blue!80!black} }
\usepackage{adjustbox}
\usepackage{diagbox}
\usepackage{xcolor}
\usepackage{eucal}

\numberwithin{equation}{section}

\theoremstyle{definition}
\newtheorem{definition}[equation]{Definition}
\newtheorem{remark}[equation]{Remark}
\newtheorem{example}[equation]{Example}
\newtheorem{question}[equation]{Question}
\newtheorem{notation}[equation]{Notation}

\theoremstyle{plain}
\newtheorem{corollary}[equation]{Corollary}
\newtheorem{proposition}[equation]{Proposition}
\newtheorem{lemma}[equation]{Lemma}
\newtheorem{conjecture}[equation]{Conjecture}
\newtheorem{theorem}[equation]{Theorem}

\newcommand{\Fun}{\text{Fun}}
\newcommand{\Cat}{\text{Cat}}
\newcommand{\cat}{\Cat_\infty}
\newcommand{\te}[1]{{\textnormal{#1}}}
\newcommand{\minus}{-}
\newcommand{\op}{\textup{op}}
\newcommand{\C}{\mathcal{C}}
\newcommand{\D}{\mathcal{D}}
\newcommand{\lax}{\textup{lax}}
\newcommand{\Corr}{\textup{Corr}}
\newcommand{\fin}{\textup{Fin}}
\renewcommand{\S}{\mathcal{S}}
\newcommand{\Sch}{\textup{Sch}}

\title{Uniqueness of Six-Functor Formalisms}

\author{Adam Dauser, Josefien Kuijper}

\begin{document}
	\maketitle
    \begin{abstract}
        We present an alternative formulation of Scholze’s notions of cohomologically proper and cohomologically étale with respect to an abstract six-functor formalism. These conditions guarantee canonical isomorphisms between the direct and exceptional direct images for certain “proper” morphisms, and between the inverse and exceptional inverse images for certain “étale” morphisms. Using this framework, we prove Scholze’s conjecture, showing that a six-functor formalism with sufficiently many cohomologically proper and étale morphisms is uniquely determined by the tensor product and inverse image functors, and can be obtained by a construction of Liu-Zheng and Mann. Additionally, we show that a generalisation of the conjecture fails, and propose a measure of this failure in terms of $K$-theory. 
    \end{abstract}
    \tableofcontents
    
  \section{Introduction}

  Many cohomology theories, such as Betti or \'etale cohomology, the cohomology of solid quasicoherent sheaves or of ind-coherent sheaves, are most naturally viewed as part of a six-functor formalism. We illustrate this throughout this introduction using \'etale cohomology. Let $\mathcal{C}$ be the $\infty$-category of geometric objects, whose cohomology a given theory computes; e.g. the category of schemes in the case of étale cohomology. For every object $X$ of $\mathcal{C}$, a six-functor formalism provides $\mathcal{D}(X)$, the symmetric monoidal $\infty$-category of all possible coefficient objects, together with restriction maps $f^*: \mathcal{D}(Y)\to \mathcal{D}(X)$. In the example of étale cohomology, $\D(X)$ is the stable $\infty$-category of complexes of \'etale sheaves on $X$ with constructible cohomology. For each map $f: X\to Y$, there is a relative notion of cohomology formalized by the right-adjoint functor $f_*: \mathcal{D}(X)\to \mathcal{D}(Y)$ to $f^*$, which recovers the cohomology of $X$ as $g_*1_X$, where $g: X\to \te{pt}$ and $1_X$ is the monoidal unit of $\mathcal{D}(X)$. Further, for a certain class $E$ of maps, there is a relative notion of cohomology with compact supports provided by a functor $f_!: \mathcal{D}(X)\to \mathcal{D}(Y)$. In \'etale cohomology, this can be defined for separated maps of finite type.   All this comes equipped with a whole host of compatibility data providing base change and projection formulas. 
  
  In practice, six-functor formalisms are almost always constructed using the following ``algorithm'', which was first executed in \cite{SGA4} but generally formalized in \cite{LZa} and \cite{LZb}.
    \begin{enumerate}
        \item Define the derived category of coefficents and the restriction functors between them as a functor $\mathcal{D}^*: \mathcal{C}^{\te{op}}\to \te{CAlg}(\te{Cat}_\infty)$;
        \item Find an appropriate class of maps $P$, where $f_*$ satisfies\footnote{This is now a property instead of additional data.} base change and the projection formula. E.g. the class of proper maps of schemes. Find another class $I$ of maps, where $f^*$ has a left-adjoint $f_\natural$ instead, which satisfies analogous conditions. E.g. open immersions of schemes;
        \item Define $f_!=j_\natural \bar{f}_*$ for any factorization $X\overset{j}{\to} \overline{X}\overset{\bar{f}}{\to} Y$, where $j\in I$ and $\bar{f}\in P$. Therefore in the resulting six-functor formalism, one necessarily has isomorphisms $f_!\cong f_*$ for any $f$ in $P$, and $j^*\cong j^!$ for $j$ in $I$.   
    \end{enumerate} 

    Liu and Zheng construct an intricate $\infty$-categorical machine to formalize one particular way to construct a six-functor formalism, such that $f_!$ and all proper base change or projection formulas associated to it do not depend on this factorization. For a concise exposition, see also \cite[Proposition A.5.10]{mann_thesis}. We will refer to this as \emph{the Liu-Zheng construction}. One can now ask the following question.

    \begin{question}
        Under which conditions and in which sense is this construction unique?
    \end{question}

 Before we further investigate this question, we recall a more formal treatment of the data that constitutes a six-functor formalism. Formally, we consider a six-functor formalism as a lax symmetric monoidal functor
$$\D:\Corr(\C,E)^\otimes \to \cat$$
for $(\C,E)$ some \textit{geometric set-up}, as in \cite[Appendix A.5]{mann_thesis}. Here $\C$ is any small $\infty$-category with limits, and $E$ a set of homotopy classes of morphisms that is closed under composition and pullbacks, and $\Corr(\C,E)^\otimes$ is the symmetric monoidal $(\infty,1)$-category of correspondences in $\C$, where the forward-pointing arrows are required to be in $E$. Such a functor $\D$ explicitly encodes a functor $f^*:\D(Y) \to \D(X)$ for every morphism $f:X\to Y$ in $\C$, as the image of the correspondence
$$Y \xleftarrow{f} X \xrightarrow{=} X $$
and for $ g\in E $, a functor $g_!:\D(X) \to \D(Y)$ as image of the morphism 
$$X \xleftarrow{=} X \xrightarrow{g} Y. $$
Since $\D$ is lax symmetric monoidal, the diagonal morphism $\Delta$ induces a symmetric monoidal structure on $\D(X)$:
$$\D(X)\times D(X) \to \D(X\times X) \xrightarrow{\Delta^*} \D(X)$$
We assume moreover that $f^*$ and $g_!$ have right adjoints $f_*$ resp. $g^!$, and the induced symmetric monoidal structure on $\D(X)$ is closed. The functoriality of  $\D$ provides a pullback square 
	\begin{center}
		\begin{tikzcd}
			X'\arrow[r, "\bar{f}"]\arrow[d, "\bar{g}"]& Y'\arrow[d, "g"]\\
			X\arrow[r, "f"] & Y
		\end{tikzcd}
	\end{center}
 with a natural  \textit{base change}  isomorphism $g^*f_! \cong\bar{f}_!\bar{g}^*$ (see also \cite[Proposition A.5.8]{mann_thesis}).  Similarly, the symmetric monoidal structure on $\Corr(\C,E)$ combined with the functoriality of  $\D$ gives a natural isomorphism 
$f_!(f^*A \otimes B) \cong A \otimes f_!B$ for all $A$ in  $\D(Y)$ and $B$ in  $\D(X)$ known as the \textit{projection formula}.
Note that the existence of natural isomorphisms $f_*\cong f_!$ or $f^* \cong f^!$ for certain morphisms is not part of this data, although they should be present in any six-functor formalism that is a result of the Liu-Zheng construction. 

\subsection{Uniqueness}
For \emph{any} six-functor formalism
$$\D:\Corr(\C,E)^\otimes \to \cat$$
there are intrinsically defined classes $P\subseteq E$ of \emph{cohomologically proper} morphisms and $I\subseteq E$ of  \emph{cohomologically étale} morphisms satisfying almost all the conditions required for the Liu-Zheng construction, apart from the condition that all morphisms in $E$ can be factored into morphisms in $I$ and $P$. In particular, there are canonical isomorphisms $p_*\cong p_!$ for $p$ in $P$ and $j^*\cong j^!$ for $j$ in $J$, or equivalently, adjunctions $p^* \vdash p_!$ and $j_!\vdash j^*$. An important caveat is, that the classes $P$ and $I$ are inductively defined, and require the morphism to be \emph{$n$-truncated for some integer $ n\geq -2 $}.

This begs the following question:

\begin{conjecture}[Scholze]
	Assume every morphism $f$ in $ E $ is $n$-truncated for some $ n $ possibly depending on $ f $, and that there are classes $I\subseteq E$ and $P\subseteq E$ such that every morphism in $I$ factors as a morphism in $I$ followed by a morphism in $P$. Are the six-functor formalisms that the Liu-Zheng construction produces exactly those, where every morphism in $ P $ is cohomologically proper and every morphism in $I$ is cohomologically \'etale?
\end{conjecture}

In Definition \ref{defn:Pr_Et}, the properties ``cohomologically proper" resp. ``cohomologically \'etale" \cite[Definition 6.10, 6.12]{scholze_notes} are carefully reformulated in terms of Beck-Chevalley conditions. This is proven to be equivalent to Scholze's definition in Proposition \ref{Different_definitions}. Herein lies the key to our proof of the conjecture, as the Liu-Zheng construction likewise, after applying the exceptionally technical Theorem \cite[Theorem 5.4]{LZb}, reduces to passing to adjoints of certain Beck-Chevalley squares. 

We call a geometric set-up $(\C,E)$, where all morphisms in $E$ are $ n $-truncated for some possibly variably $n$, with classes of morphisms $I,P\subseteq E$ suitable for the Liu-Zheng construction a \textit{Nagata set-up}.  A six-functor formalisms, that satisfies the reformulated conditions suggested by Scholze's conjecture is called a \textit{Nagata six-functor formalism} (see Definition \ref{defn:nagata} and Corollary  \ref{cor:same_as_nagata}). In this formulation, Scholze's conjecture becomes:

\begin{theorem}[Theorem \ref{theorem:scholzes conjecture}]
	Let $ (\mathcal{C}, I, P, E) $ be a Nagata set-up. Restriction along $ \mathcal{C}^{\te{op}, \sqcup} \to \te{Corr}(\mathcal{C}, E)^\otimes$ induces an equivalence of categories between the $ \infty $-category of Nagata six-functor formalisms, and the subcategory of functors $ \mathcal{D}^*:\mathcal{C}^\te{op}\to \te{CAlg}(\te{Cat}_\infty) $ satisfying the conditions required by the Liu-Zheng construction of a six-functor formalism and maps $\alpha: \mathcal{D}\to \mathcal{D}' $ such that for every morphism $ f: Y\to X$ in $ P $ resp. in $ I $, the diagram
		\begin{center}
			\begin{tikzcd}
					\mathcal{D}(X)\arrow[d, "f^*"]\arrow[r, "\alpha_X"] & \mathcal{D}'(X)\arrow[d, "f^*"]\\
					\mathcal{D}(Y)\arrow[r, "\alpha_Y"] & \mathcal{D}'(Y)
			\end{tikzcd}
		\end{center}
	is vertically right resp. left adjoinable.
\end{theorem}

\subsection{Other formalisations of six-functor formalisms}

This paper is somewhat orthogonal to the $(\infty, 2)$-categorical approach to six-functor formalisms. In \cite{gaitsgory_rozenblyum}, six-functor formalisms are defined in terms of functors from the symmetric monoidal $ (\infty, 2) $-category of correspondences $ \te{Corr}(\mathcal{C})_{\te{all}, \te{all}}^\te{proper} $, which also explicitly encode the adjunction $ f^*\dashv f_! $ for $ f $ in $ P=\{\te{proper maps}\} $ and an isomorphism $ f^*f_!\simeq \te{id} $, which is assumed to be the counit of adjunction $ f_!\dashv f^*$ for $ I=\{\te{open immersions}\} $. 

While it is important for their construction that maps in $ I $ are $ -1 $-truncated, proper morphisms of derived schemes are only $ n $-truncated for any $n$, if they are finite. Still, their construction of a six-functor formalism in their sense is via a universal property of $ \te{Corr}(\mathcal{C})_{\te{all}, \te{all}}^\te{proper} $ so, in particular, automatically unique. 

Even if only all maps in $I\cap P$ are $n$ truncated for some $n$ possibly depending on the map, there is an analogously defined symmetric monoidal $(\infty, 2)$-category $\te{Corr}(\mathcal{C})_{\te{all}, E}^{P, I}$. In \cite{cnossen} it is shown that lax symmetric monoidal functors $\te{Corr}(\mathcal{C})_{\te{all}, E}^{P, I}\to \te{Cat}_\infty$ are equivalent to functors $\mathcal{C}^\te{op}\to \te{CAlg}(\te{Cat}_\infty)$ satisfying the same conditions as the Liu-Zheng construction. It is believed, that this recovers the Liu-Zheng construction in general. Under the assumption, that every morphism in $E$ is $n$-truncated for some $n$ possibly depending on the morphism, this is a consequence of the main theorem of this paper.

Such general uniqueness results are possible in the $(\infty, 2)$-categorical setting as they do not have to reconstruct the adjunctions $ f^*\dashv f_! $ for $f\in P$ inductively as we do, but have them encoded as additional data yielding better rigidity. This excludes twists as defined in Section \ref{sect:K-theory} as for twisted six-functor formalisms no coherent morphisms $ f_{!'}\to f_* $ can exist.\footnote{This would amount to a coherent system of maps $ 1_X\to \mathcal{L}_{X/Y} $ trivializing the twists $ \mathcal{L}_{X/Y} $ and all their coherence isomorphisms.}

Our approach can, however, be directly compared to  \cite{ayoub_i} and \cite{ayoub_ii}, \cite{Cisinski_deglise}, \cite{drew_gallauer}, and \cite{khan}, who additionally assume smooth base-change formulas to hold. A precise translation to the framework of coefficient systems is provided in Section \ref{subsect:coeff}, leading to the characterization of $\mathcal{SH}$ as initial object of a well-chosen subcategory of the $\infty$-category of Nagata six-functor formalisms. In Section \ref{subsect:khan} we spell out how a proof of \cite[Theorem 2.51]{khan}, independent of $(\infty,2)$-categories, can be obtained.


\subsection{Relation to $K$-theory}
    In the  proving uniqueness of the construction it is indispensable to assume that every morphism in $E$ is $n$-truncated for some $n$, possibly depending on $f$. This allows for certain inductive arguments.
    
    In Section \ref{sect:K-theory} we show why this condition is necessary. Example \ref{Euler-characteristics} constructs a six-functor formalism on the opposite category of finite anima that satisfies all other assumptions of Theorem \ref{theorem:scholzes conjecture} but cannot be reconstructed.

    The obstruction is that $|\te{Corr}(\mathcal{C})|$ is no longer contractible. This allows us to construct a class of six-functor formalisms of the form
        \begin{align*}
            \mathcal{D}_\mathcal{L}: \te{Corr}(\mathcal{C})^\otimes\to |\te{Corr}(\mathcal{C})|^\otimes \to \te{Cat}_\infty.
        \end{align*}
    	 The Liu-Zheng construction applied to $\mathcal{D}_\mathcal{L}^*: \mathcal{C}^{\te{op}, \sqcup}\to \te{Cat}_\infty$ and any sets of morphisms $I$ and $P$ on the other hand would always reconstruct the constant six-functor formalism $X\mapsto \mathcal{D}(1)$. 
    
    
    In Definition \ref{def:twist}, we propose a general notion of twists for any three-functor formalism $ \mathcal{D}:\te{Corr}(\mathcal{C})^\otimes\to \te{Cat}_\infty $\footnote{In Remark \ref{def:generaltwists} this notion is extended to all three-functor formalisms $ \mathcal{D}: \te{Corr}(\mathcal{C}, E)^\otimes\to \te{Cat}_\infty $.}; they are given by strong symmetric monoidal functors
    	\begin{align*}
    		\mathcal{L}_{X/-}: T_\mathcal{C}(X):=\Omega |\te{Corr}(\mathcal{C}_{X/})^\otimes|\to \mathcal{D}(X)^\otimes
    	\end{align*}
	functorially varying over $ X\in \mathcal{C}^{\te{op}, \sqcup} $. Remark \ref{explanation:twists} explains how to extract a $ \otimes $-invertible object $ \mathcal{L}_{X/Y}\in \mathcal{D}(X) $ for any morphism $ f: X\to Y $ and all higher coherence data such that 
		\begin{align*}
			f_{!'}(-)= f_!(-\otimes \mathcal{L}_{X/Y})
		\end{align*}
	\emph{should} yield a three-functor formalism $ \mathcal{D}_\mathcal{L}: \te{Corr}(\mathcal{C})^\otimes\to \te{Cat}_\infty $ which agrees with $ \mathcal{D} $ on $ \mathcal{C}^{\te{op}, \sqcup} $ making it inaccessible to the Liu-Zheng construction.
	We \emph{warn} the reader that we don't provide a formal construction of this twisted three-functor formalism in this article in this generality.

    Taking these geometric realisations has striking similarities with the Quillen Q construction of connective $K$-theory. Using Lemma \ref{lem:QuillenQ}, one can even directly compare the two:
        \begin{align*}
            T_\mathcal{C}(X)=K(\te{Stab}(\mathcal{C}_{X/,/X})).
        \end{align*}
    Using this description, we may define twists of six-functor formalisms for all geometric set-ups (see Remark \ref{def:generaltwists}).
    For six-functor formalisms on the $\infty$-category of affine $ \mathbb{E}_\infty $-schemes over some base, where $E$ is the class of morphisms of finite presentation, this allows us to calculate
        \begin{align*}
            T_{(\mathcal{C}, E)}(\te{Spec}(R))=K(R),
        \end{align*}
    see Proposition \ref{prop:twists on affine schemes}.
    Spelled out, twists of six-functor formalisms on $ \mathbb{E}_\infty $-schemes are given by coherent families of representations of the connective algebraic $K$-theory spectra. 
    
    One example is fashioned by the determinant map $ \det: K(R)\to \te{Pic}(\te{Spec}(R)) $. Given $ \mathcal{D} $ a six-functor formalism of ind-coherent sheaves, solid quasi-coherent sheaves or the like, we can embed $ \te{Pic}(\te{Spec}(R)) $ into $ \mathcal{D}^*(\te{Spec}(R)) $. The resulting six-functor formalism twisted by the composition of these maps should have exceptional pushforwards
    	\begin{align*}
    		f_{!'}(-)=f_!(-\otimes \det (\mathbb{L}_f)),
    	\end{align*}
	where $ \mathbb{L}_f $ is the relative cotangent complex (see Example \ref{great_example}).
    
    \subsection{Outlook and conjectures}
    
    	Pertaining to cases not covered by Scholze's conjecture, we present a conjecture characterising the six-functor formalisms $ \mathcal{D}':\te{Corr}(\mathcal{C})^\otimes\to \te{Cat}_\infty $, that come from twisting $ \mathcal{D}:\te{Corr}(\mathcal{C})^\otimes\to \te{Cat}_\infty $. In particular, this conjecture asks for a construction of the twisted three-functor formalisms as sketched in Remark \ref{explanation:twists}.
    	
    	\begin{conjecture}[Conjecture \ref{conjecture:LZ}]
    		Let $\mathcal{D}: \te{Corr}(\mathcal{C})^\otimes\to \te{Cat}_\infty$ a three-functor formalism. Denote by $\mathfrak{LZ}_\mathcal{D}^\otimes$ its symmetric monoidal $(\infty, 2)$-category of cohomological correspondences with underlying symmetric monoidal $\infty$-category $\te{LZ}_\mathcal{D}^\otimes$. There is a canonical isomorphism of anima
    		\begin{align*}
    			\te{Hom}_{\te{CAlg}(\te{Cat}_{\infty})_{ \mathcal{C}^{\te{op}, \sqcup}/}}(\te{Corr}(\mathcal{C})^\otimes, \textnormal{LZ}_\mathcal{D}^\otimes)\simeq
    			\te{Hom}_{\te{Fun}(\mathcal{C}^\te{op}, \te{CAlg}(\te{Cat}_\infty))}(T_\mathcal{C}, \mathcal{D}^*).
    		\end{align*}
    	\end{conjecture} 
    
    	The Liu-Zheng construction and, in particular, theorem \cite[Theorem 5.4]{LZb} on which this article depends heavily has not appeared in the peer-reviewed literature as of 2025. Conjecture \ref{conjecture:LZ} is entirely independent and would imply that any construction of this flavour is necessarily unique provided the assumption on truncatedness of morphisms in $ E $ holds.

\subsection{Notation}
    We recall some notation used in \cite{LZa} and \cite{LZb}.
    \begin{notation}[{\cite[Definition 6.1.1, Lemma 6.1.2]{LZa}}]
        For $n\geq 0$, we denote by $C(\Delta^n)$ the full subcategory of $\Delta^n\times (\Delta^n)^\op$ on objects $(i,j)$ such that $i\leq j$. 

        Let $\mathcal{C}$ be an $\infty$-category with finite limits, and $E$ a set of homotopy classes of morphisms in $\C$, closed under composition and pullbacks by arbitrary morphisms. Then we denote by $\Corr(\C,E)$ the quasi-category whose $n$-simplices are given by functors 
        $$C(\Delta^n)\to \C $$
        that send maps $(i,j) \to (i+1, j)$, for $i$ in $\Delta^n$ and $j$ in $(\Delta^n)^\op$, to morphisms in $E$, and that send squares of the form
        \begin{center} \begin{tikzcd} (i,j) \ar[r] & (i+1,j) \\
        (i,j+1) \ar[u] \ar[r] & (i+1,j+1)\ar[u]
        \end{tikzcd}
        \end{center}for $i+1\leq j$ to pullback squares. In the case that $E$ is the set $All$ of all morphisms in $\C$, we denote $\Corr(\C,E)$ by $\Corr(\C)$.
    \end{notation}
    In other words, an edge in $\Corr(\C,E)$ can be thought of as a span 
    $$X \xleftarrow{f} Z \xrightarrow{g} Y$$
   in $\C$, where $g$ is in $E$. The composition of to spans $X \leftarrow Z \rightarrow Y$ and $Y \leftarrow V \rightarrow W$ is given by the composite $X \leftarrow Z \leftarrow Z\times_Y V \rightarrow V \rightarrow W$.\\

\begin{notation}[{\cite[Definition 3.16]{LZb}}]\label{defn:multisimplicial_nerve} 
Let $\C$ an $\infty$-category with finite limits and $I_1,\dots, I_k$ sets of morphisms in $\C$ or $\C^\op$. Then we denote by $\C(I_1,\dots,I_k)$ the $k$-simplicial set whose $(n_1,\dots, n_k)$-simplices are $n_1$-by-$n_2$-by-$\dots$-by-$n_k$ $k$-dimensional hypercubes in $I_1\times \dots \times I_k$, where all squares are cartesian squares in $\C$.

We denote by $\delta_k^* \C(I_1,\dots,I_k)$ its diagonal, in other words, the simplicial set whose $n$-simplices are $(n,\dots,n)$-simplices in $\C(I_1,\dots,I_k)$.

 We use $\C^\textup{cart}$ as a shorthand for the bisimplicial set $\C(All,All)$.
\end{notation}

Consider a square of categories
\begin{center}\begin{tikzcd}
    \mathcal{A} \ar[r, "f"] \ar[d, "g"] & \mathcal{B} \ar[d, "j"] \\
    \mathcal{C} \ar[r, "h"] & \mathcal{D}
\end{tikzcd} \end{center} that commutes up to a specified isomorphism $ jf \cong hg$.
Given right adjoints $g^R$ and $j^R$ of $g$ and $j$ respectively, the co-unit and unit of these adjunctions induce a canonical morphism 
$$fg^R \to j^Rjfg^R \cong j^Rhgg^R \to j^R g.$$
Similarly, given left adjoints $g^L$ and $j^L$ of $g$ and $j$  respectively, there is a canonical morphism
$$j^Lh\to j^Lhgg^L \cong j^Ljfg^L \to fg^L.$$
   \begin{definition}
       We say that a square of categories 
   \begin{center}\begin{tikzcd}
    \mathcal{A} \ar[r, "f"] \ar[d, "g"] & \mathcal{B} \ar[d, "j"] \\
    \mathcal{C} \ar[r, "h"] & \mathcal{D}
\end{tikzcd} \end{center}
is \textit{vertically right adjoinable} if $g$ and $j$ have right adjoints $g^R$ and $j^R$ respectively, and if the canonical morphism  $fg^R\to j^R g$ is an isomorphism. We say that such a square is \textit{vertically left adjoinable} if $g$ and $j$ have left adjoints $g^L$ and $j^L$ respectively, and if the canonical morphism $j^Lh\to fg^L$ is an isomorphism. We define \textit{horizontally} left and right adjointable squares analogously.

We say that a square of $\infty$-categories is vertically/horizontally left/right adjoinable if the square of homotopy categories is.
   \end{definition}

\subsection{Acknowledgements}

    We would like to thank Ko Aoki and Peter Scholze for useful discussions. We thank Martin Gallauer, Gijs Heuts, and Adeel Khan for comments on a previous version of this manuscript.  
    The first named author acknowledges funding from the DFG Leibniz Preis through Peter Scholze. The second named author was partially supported by the ERC grant ERC-2017-STG 759082 through Dan Petersen.

 \section{Cohomologically proper and étale morphisms}\label{sect:cohom_proper_etale}
		We recall some notions from \cite{mann_thesis}, \cite{LZa} and \cite{LZb}. 
		
		\begin{definition}\label{defn:nagata_setup}
			A \emph{geometric set-up} $ (\mathcal{C}, E) $ is a small $ \infty $-category together with a set of homotopy classes of morphisms $ E $, such that
				\begin{enumerate}
					\item pullbacks of morphisms in $E$ along arbitrary morphisms exist, and are in $E$;
					\item $ E $ is stable under composition;
					\item $ E $ contains all isomorphisms.
				\end{enumerate}
			A \emph{Nagata set-up} $ (\mathcal{C}, E, I, P) $ is a geometric set-up $ (\mathcal{C}, E) $ together with two subsets $ I, P\subseteq E $, such that
				\begin{enumerate}
                    \item any morphism $ f\in E $ is a composite $ \bar{f}\circ j $ for $ j\in I $ and $ \bar{f}\in P $;
					\item $ (\mathcal{C}, I) $ and $ (\mathcal{C}, P) $ are geometric set-ups;
					\item given $ f: X\to Y $ in $ \mathcal{C} $ and $ g: Y\to Z $ in $ I $, we have $ f \in I $ if and only if $ g\circ f \in I $;
					\item given $ f: X\to Y $ in $ \mathcal{C} $ and $ g: Y\to Z $ in $ P $, we have $ f \in I $ if and only if $ g\circ f \in P $;
     \item for any morphism $f\in I\cup P$, there exists some $n\geq -2$ such that $f$ is $n$-truncated.
     
     \end{enumerate}
			
		\end{definition}
	\begin{notation}
	    Given an $\infty$-category $\C$, we let $\C^\sqcup$ denote the $\infty$-operad of \cite[Construction 2.4.3.1]{HA}.
	\end{notation}
	We observe that an edge $f:(X_i)_I\to (Y_j)_J$ in $((\C^\op)^\sqcup)^\op$, henceforth denoted $\C^{\op,\sqcup,\op}$, consists of a map $\alpha:J \to I$ in $\fin_*$, and for every $j\in J$ with $\alpha(j)=i$ a map $X_i\to Y_j$. 
\begin{notation}
  Given a set of edges $E$ in $\C$, we denote by $E_\sqcup$ the set of edges in $\C^{\op,\sqcup,\op}$ whose components $X_i\to Y_j$ are all in $E$. We denote by $E_-\subseteq E_\sqcup$ the subset of edges that lie over an identity morphism in $\fin_*$. 
\end{notation}
Given a geometric set-up $(\C,E)$, it follows that $(\C^{\op,\sqcup,\op}, E_\minus)$ is a geometric set-up. Likewise, given a Nagata set-up $ (\mathcal{C}, E, I, P ) $, it follows that $(\C^{\op,\sqcup,\op},E_-,I_-,P_-)$ is a Nagata set-up.

For $(\C,E)$ a geometric set-up, by \cite[Proposition 6.1.3]{LZa} we have that $\te{Corr}(\C^{\op,\sqcup,\op}, E_\minus) $ is a symmetric monoidal $\infty$-category whose underlying $\infty$-category is $\Corr(\C,E)$. Henceforth, we denote $\te{Corr}(\C^{\op,\sqcup,\op}, E_\minus) $ by $\Corr(\C,E)^\otimes$.

\begin{definition} Let $ (\mathcal{C}, E) $ be a geometric set-up.
			\begin{enumerate}
					\item A \emph{two-functor formalism} is a functor $ 	\mathcal{D}: \te{Corr}(\mathcal{C}, E)\to \te{Cat}_\infty $. 
					\item A \emph{three-functor formalism} is a lax cartesian structure $\mathcal{D}: \te{Corr}(\mathcal{C}, E)^\otimes\to \te{Cat}_\infty. $ (see \cite[Definition 2.4.1.1.]{HA})
					\item A \emph{six-functor formalism} is a three-functor formalism
					$ 	\mathcal{D}: \te{Corr}(\mathcal{C}, E)^\otimes\to \te{Cat}_\infty $ taking values in the $\infty$-category $ \te{Cat}_\infty^\te{L} $ of $\infty$-categories and functors that are left-adjoints.
			\end{enumerate}
			For $\D$ a two-functor formalism, we denote the image of $ (f: X\to Y )\in E $ under $ \te{Corr}(\mathcal{C}, E)$ $ \to \te{Cat}_\infty  $ by $ f_!: \mathcal{D}(X)\to \mathcal{D}(Y) $ and the image of $ (f: X\to Y)\in \mathcal{C} $ under $ \mathcal{C}^{\te{op}}\to\te{Corr}(\mathcal{C}, E)\to  \te{Cat}_\infty $ by $ f^*: \mathcal{D}(Y)\to\mathcal{D}(X) $. For $\D$ a three-functor formalism, the lax cartesian structure $ \mathcal{C}^{\te{op}, \sqcup} \to \te{Corr}(\mathcal{C}, E)^\otimes\to \te{Cat}_\infty $ corresponds to a functor $ \mathcal{C}^\te{op}\to \te{CAlg}(\te{Cat}_\infty) $ via \cite[Theorem 2.4.3.18.]{HA}. We denote the corresponding symmetric monoidal structure on $ \mathcal{D}(X) $ by $ \otimes $ with unit $ 1_X $. 
		\end{definition}

		\begin{remark}
			We will not need to treat three- and six-functor formalisms separately. Using
			 $ \te{Corr}(\mathcal{C}, E)^\otimes=\te{Corr}(\C^{\op,\sqcup,\op}, E_\minus) $, we view them as a subclass of two-functor formalisms for the geometric set-up $ (\C^{\op,\sqcup,\op}, E_\minus)  $.
		\end{remark}
	In \cite{scholze_notes} the notions of  ``cohomologically proper" and ``cohomologically \'etale'' are defined for a three-functor formalism. We will generalize these notions to a two-functor formalism.
\begin{definition}\label{defn:Pr_Et}
    Given a two-functor formalism $\D:\te{Corr}(\C,E) \to \cat$, we say that a morphism $f:X\to Y$ in $E$ satisfies
\begin{itemize}
 \item[(Pr-I)] if $f$ is $n$-truncated, and
\begin{itemize}
    \item[(1)] either $n=-2$, in which case $f_!$ has a left-adjoint $f^\natural$; in which case the square
        \begin{center}
            \begin{tikzcd}
                \mathcal{D}(X)\ar[d, "f_!"] \ar[r, "="]& \mathcal{D}(X) \ar[d, "="] \\
                \mathcal{D}(Y) \ar[r, "f^*"] &\mathcal{D}(X)
            \end{tikzcd}
        \end{center}
        is automatically vertically left-adjoinable, giving a natural isomorphism
$$\iota_{-2}:f^*\to f^* f_!f^\natural \cong f^\natural,$$ 
        \item[(2)] or $n\geq -1$ and $\Delta_f$ satisfies (Pr-I), with a natural isomorphism $\iota_{n-1}:\Delta_{f}^*\xrightarrow{\cong} \Delta_{f}^\natural$, and the square
\begin{center}
    \begin{tikzcd}
        \mathcal{D}(X)\arrow[d, "f_!"] \arrow[r, "p_2^*"] & \mathcal{D}(X \times_Y X) \arrow[d, " p_{1,!}"] \\
        \mathcal{D}(Y)\arrow[r, "f^*"] & \mathcal{D}(X)
    \end{tikzcd}
\end{center}
is vertically left-adjoinable, giving a natural isomorphism 
$$f^* \cong \Delta_f^\natural p_1^\natural f^* \xrightarrow[BC]{\cong} \Delta_f^\natural p_2^* f^\natural \xrightarrow[\iota_{n-1}^{-1}]{\cong}\Delta_f^* p_2^* f^\natural \cong f^\natural.$$ 
\end{itemize} 
\item[(Pr-II)] if $f$ is $n$-truncated, and
\begin{itemize}
    \item[(1)] either $n=-2$, in which case $f^*$ has a right-adjoint $f_*$; in which case the square \begin{center}
        \begin{tikzcd}
            \D(X) \arrow[r, "f_!"] \arrow[d, "=" ] & \D(Y) \arrow[d, "f^*"] \\
            \D(X) \arrow[r, "="] & \D(X)
        \end{tikzcd}
    \end{center} 
    is automatically vertically right-adjoinable, giving a natural isomorphism 
    $$\epsilon_{-1}:f_! \to f_*f^*f_! \cong f_*,$$ 
    \item[(2)] or $n\geq 1$ and $\Delta_f$ satisfies (Pr-II), with a natural isomorphism $\epsilon_{n-1}:\Delta_{f,!}\xrightarrow{\cong} \Delta_{f,*}$, and the square
    \begin{center}
        \begin{tikzcd}
            \D(X) \arrow[r, "f_!"] \arrow[d, "p_2^*"] & \D(Y) \arrow[d, "f^*"]\\
            \D(X \times_Y X)\arrow[r, "p_{1,!}"] & \D(X)
        \end{tikzcd}
    \end{center}
    is vertically right-adjoinable, giving a natural isomorphism
    $$\epsilon_n:f_! \cong f_!p_{2,*}\Delta_{f,*} \xrightarrow[BC]{\cong} f_*p_{1,!}\Delta_{f,*} \xrightarrow[\epsilon_{n-1}^{-1}]{\cong} f_* p_{1,!} \Delta_{f,!} \cong f_*.$$ 
\end{itemize}
\end{itemize}
We define conditions (Ét-I) and (Ét-II) dually; a morphism $f:X\to Y$ in $E$ satisfies
\begin{itemize}
    \item[(Ét-I)] if $f$ is $n$-truncated, and 
    \begin{itemize}
        \item[(1)] either $n=-2$, in which case $f_!$ has a right-adjoint $f^!$; in which case the square   \begin{center}
            \begin{tikzcd}
                \mathcal{D}(X)\ar[d, "f_!"] \ar[r, "="]& \mathcal{D}(X) \ar[d, "="] \\
                \mathcal{D}(Y) \ar[r, "f^*"] &\mathcal{D}(X)
            \end{tikzcd}
        \end{center}
        is automatically vertically right-adjoinable, giving a natural isomorphism
        $$\iota_{-2}:f^! \cong f^*f_! f^! \to f^*, $$
        \item[(2)] or $n\geq -1$ and $\Delta_f$ satisfies (Ét-I), with a natural isomorphism $\iota_{n-1}:\Delta_{f}^!\xrightarrow{\cong} \Delta_{f}^*$, and the square
\begin{center}
    \begin{tikzcd}
        \mathcal{D}(X)\arrow[d, "f_!"] \arrow[r, "p_2^*"] & \mathcal{D}(X \times_Y X) \arrow[d, " p_{1,!}"] \\
        \mathcal{D}(Y)\arrow[r, "f^*"] & \mathcal{D}(X).
    \end{tikzcd}
\end{center}
is vertically right-adjoinable, giving a natural isomorphism 
$$\iota_n:f^!\cong \Delta_f^*p_2^*f^!\xrightarrow[BC]{\cong} \Delta^*_f p_1^!f^*\xrightarrow[\iota_{n-1}^{-1}]{\cong} \Delta^!_fp_1^!f^* \cong f^*.$$
\end{itemize} 
\item[(Ét-II)] if $f$ is $n$-truncated, and
\begin{itemize}
    \item[(1)] either $n=-2$, in which case $f^*$ has a left-adjoint $f_\natural$; in which case the square \begin{center}
        \begin{tikzcd}
            \D(X) \arrow[r, "f_!"] \arrow[d, "=" ] & \D(Y) \arrow[d, "f^*"] \\
            \D(X) \arrow[r, "="] & \D(X)
        \end{tikzcd}
    \end{center} 
    is automatically vertically left-adjoinable, giving a natural isomorphism
$$\epsilon_{-2}:f_\natural \cong  f_\natural f^*f_! \to f_!,$$ 
    \item[(2)] or $n\geq 1$ and $\Delta_f$ satisfies (Ét-II), with a natural isomorphism $\epsilon_{n-1}:\Delta_{f,\natural}\xrightarrow{\sim} \Delta_{f,!}$, and the square
    \begin{center}
        \begin{tikzcd}
            \D(X) \arrow[r, "f_!"] \arrow[d, "p_2^*"] & \D(Y) \arrow[d, "f^*"]\\
            \D(X \times_Y X)\arrow[r, "p_{1,!}"] & \D(X)
        \end{tikzcd}
    \end{center}
    is vertically left-adjoinable, giving a natural isomorphism
    $$\epsilon_n:f_\natural \cong f_\natural p_{1,!}\Delta_{f,!} \xrightarrow[BC]{\cong} f_!p_{1,\natural}\Delta_{f,!}\xrightarrow[\epsilon_{n-1}^{-1}]{\cong} f_! p_{1,\natural} \Delta_{f,\natural} \cong f_!  .$$ 
\end{itemize}
\end{itemize}  
\end{definition}
We say that a morphism $f$ satisfies one of these conditions \textit{stably} if the base change of $f$ along any morphism satisfies the condition as well; and the same is true for its diagonal.\footnote{This should be understood as an inductive definition like before.}

\begin{remark}
	Unlike in the case of three-functor formalisms and cohomologically proper morphisms, in general two-functor formalisms, morphisms satisfying $ (\te{Pr-II}) $ or $ (\te{Pr-II}) $ do not necessarily satisfy it stably.
	
	An explicit counter example can be constructed on $\mathcal{C}=\Delta^1\times\Delta^1$. Denote the objects and morphisms in this category by 
	\begin{center}
		\begin{tikzcd}
			X' \ar[r, "\bar{g}"]\ar[d, "\bar{f}"]  & X\ar[d, "f"]\\
			Y'\ar[r, "g"] & Y.
		\end{tikzcd}
	\end{center}
	Set $E=\{ \te{id}_X, \te{id}_{X'}, \te{id}_Y, \te{id}_{Y'}, f, \bar{f} \}$ . The data of a $2$-functor formalism $\te{Corr}(\mathcal{C}, E)\to \te{Cat}_\infty$ with $\mathcal{D}(X)=\mathcal{D}(Y)=\{*\}$ is equivalent a to pointed $\infty$-category $\mathcal{D}(Y')$ with a pointed full sub-$\infty$-category $\bar{f}_!: \mathcal{D}(X')\hookrightarrow \mathcal{D}(Y')$ admitting a retraction $\bar{f}^*$. Then $f$ is cohomologically proper, but its base change $\bar{f}$ need not be. An explicit counterexample is given by $\mathcal{D}(Y')=\Delta^2$ pointed with $0$, where $\mathcal{D}(X')$ is the full subcategory on $\{0, 2\}$ and $\bar{f}^*$ sends $1$ to $2$.
\end{remark}

\begin{lemma}\label{lem:stably}
    Let $\D:\te{Corr}(\C,E) \to \cat$ be a two-functor formalism, and let $f:X\to Y$ be in $E$.
    \begin{itemize}
        \item[(1)] If $f$ satisfies (Pr-I) stably, 
        then $f$ satisfies (Pr-II);
        \item[(2)] If $f$ satisfies (Pr-II) stably, 
        then $f$ satisfies (Pr-I);
         \item[(3)] If $f$ satisfies (Ét-I) stably, 
         then $f$ satisfies (Ét-II);
        \item[(4)] If $f$ satisfies (Ét-II) stably, 
         then $f$ satisfies (É-I).
    \end{itemize}
\end{lemma}
\begin{proof}
    The arguments are symmetric, and therefore we only prove (1).  If $f$ is $\minus 2$-truncated, then $f$ satisfies (Pr-II) automatically. By induction, we assume that (1) holds for $n$-truncated morphisms, and we assume $f$ is $(n+1)$-truncated. To show that $f$ satisfies (Pr-II), it suffices to show that for any cartesian square 
    \begin{center}
					\begin{tikzcd}
						X' \ar[r, "\bar{g}"]\ar[d, "\bar{f}"]  & X\ar[d, "f"]\\
						Y'\ar[r, "g"] & Y,
					\end{tikzcd}
				\end{center}
where $f$ satisfies (Pr-I), the square 
\begin{equation}\label{eq:square}
        \begin{tikzcd}
          \D(Y') \arrow[d, "\bar{f}^*"]\arrow[r,"g_!"]&\D(Y) \arrow[d, "f^*"]\\
          \D(X')\arrow[r,"g_!"] & \D(X)
        \end{tikzcd}
    \end{equation}
    is vertically right-adjoinable. By assumption, there are natural adjunctions $f^\natural \vdash f_!$ and $\bar{f}^\natural \vdash \bar{f}_!$. Therefore it suffices to show that the square
    \begin{center}
					\begin{tikzcd}
						\D(X') \ar[r, "{\bar{g}_!}"]\ar[d, "{\bar{f}_!}"]  & \D(X)\ar[d, "f_!"]\\
						\D(Y')\ar[r, "g_!"] & \D(Y),
					\end{tikzcd}
				\end{center} 
is vertically left-adjoinable along these adjunctions, and that the resulting square is homotopy equivalent to (\ref{eq:square}). This holds for $-2$-truncated morphisms, since $f^\natural$ and $f^*$ are both quasi-inverse to $f_!$. By induction, we also assume that this claim holds for $\Delta_f$. Now we consider the commuting diagram\begin{center}
\begin{tikzcd}
& \D(X') \arrow[rr, "\bar{g}_!"] \arrow[d, "{\Delta_{\bar{f},!}}"']         &          & \D(X) \arrow[d, "{\Delta_{f,!}}"]       \\
& \D(X'\times_{Y'}X') \arrow[rr, "\bar{g}'_!"] \arrow[dd, "{\bar{p}_{1,!}}"{yshift=1.5em}] &       & \D(X\times_Y X) \arrow[dd, "{p_{1,!}}"] \\
\D(X') \arrow[ru, "\bar{p}_2^*"] \arrow[rr, "\bar{g}_!"'{xshift=-1em}] \arrow[dd, "\bar{f}_!"] &        & \D(X) \arrow[ru, "p_2^*"] \arrow[dd, "f_!"{yshift=1.5em}] &    \\
 & \D(X') \arrow[rr, "\bar{g}_!"{xshift=-1.5em}]        &          & \D(X)        \\
\D(Y') \arrow[ru, "\bar{f}^*"] \arrow[rr, "g_!"]      &                & \D(Y) \arrow[ru, "f^*"]                     &                                        
\end{tikzcd}\end{center}
From this we can form the diagram
\begin{equation}\label{eq:bench}
\begin{tikzcd}
& \D(X') \arrow[rr, "\bar{g}_!"]          &          & \D(X)   \\
& \D(X'\times_{Y'}X') \arrow[rr, "\bar{g}'_!"]  \arrow[u, "\Delta_{\bar{f}}^\natural" ] &       & \D(X\times_Y X)\arrow[u, "\Delta_f^\natural"] \\
\D(X') \arrow[ru, "\bar{p}_2^*"] \arrow[rr, "\bar{g}_!"'{xshift=-1em}] &        & \D(X) \arrow[ru, "p_2^*"] &    \\
 & \D(X') \arrow[rr, "\bar{g}_!"{xshift=-1.5em}] \arrow[uu, "{\bar{p}_{1}^\natural}"{yshift=1.5em}]        &          & \D(X)  \arrow[uu, "p_1^\natural"]      \\
\D(Y') \arrow[ru, "\bar{f}^*"] \arrow[rr, "g_!"] \arrow[uu, "\bar{f}^\natural"]     &                & \D(Y) \arrow[ru, "f^*"]  \arrow[uu, "f^\natural"{yshift=1.5em}]                    &                                          
\end{tikzcd}
\end{equation}
by taking vertical adjoints along the provided adjunction. By the induction hypothesis, since $\Delta_f$ is $n$-truncated, the upper square is homotopic to the one that involves $\Delta_{\bar{f}}^*$ and $\Delta_f^*$ and comes from functoriality of $\D$.
Since $f$ and $\bar{f}$ satisfy (Pr-I), the left and right face of the cube are filled with isomorphisms.  In the diagram
\begin{center}
   \begin{tikzcd}
    \D(X') \arrow[rr,"{\bar{g}_!}"] & & \D(X) \\
    \D(X'\times_{Y'}X') \arrow[u, "{\Delta_{\bar{f}}^\natural}"] \arrow[rr, "{\bar{g}'_!}"] & & \D(X \times_Y X) \arrow[u, "\Delta_f^\natural"]\\
    \D(X') \arrow[u, "\bar{p}_2^\natural"] \arrow[rr,"\bar{g}_!"] & & \D(X)\arrow[u, "p_2^\natural"]
\end{tikzcd} 
\end{center}
the outer square commutes and is homotopic to the one that comes from functoriality of $\D$, since we can apply the claim we are proving to the $-2$-truncated morphism $p_2\circ\Delta_f\cong \textup{id}$. Since the upper square commutes by the induction hypothesis, so does the lower square. This shows that the diagram (\ref{eq:bench}) provides an isomorphism from 
\begin{center}
    \begin{tikzcd}
        \D(X') \arrow[r, "\bar{g}_!"] & \D(X) \\
        \D(Y') \arrow[u, "\bar{f}^\natural"] \arrow[r, "g_!"] & \D(Y) \arrow[u, "f^\natural"]
    \end{tikzcd} \  \ to \ \
    \begin{tikzcd}
        \D(X') \arrow[r, "\bar{g}_!"] & \D(X) \\
        \D(Y') \arrow[u, "\bar{f}^*"] \arrow[r, "g_!"] & \D(Y), \arrow[u, "f^*"]
    \end{tikzcd}\end{center}
completing the proof.
\end{proof}
\begin{corollary}\label{cor:more_adjoinable_squares}
    Let $\D:\Corr(\C,E)\to \cat$ be a two-functor formalism, and     \begin{equation}\label{eq:again_a_generic_square}
					\begin{tikzcd}
						X' \ar[r, "\bar{g}"]\ar[d, "\bar{f}"]  & X\ar[d, "f"]\\
						Y'\ar[r, "g"] & Y,
					\end{tikzcd}
				\end{equation} a cartesian square in $\C$. 
If $f$ and $\bar{f}$ both satisfy (Pr-I) (resp. (Ét-I)) stably or both satisfy (Pr-II) (resp. (Ét-II)) stably, then the squares
        \begin{center}
					\begin{tikzcd}
						\D(X') \ar[r, "\bar{g}_!"]\ar[d, "\bar{f}_!"]  & \D(X)\ar[d, "f_!"]\\
						\D(Y')\ar[r, "g_!"] & \D(Y)
					\end{tikzcd}
				\ \ and \ \ 
					\begin{tikzcd}
						\D(X') \ar[d, "\bar{f}_!"]  & \D(X)\ar[d, "f_!"] \ar[l, "\bar{g}^*"']\\
						\D(Y') & \D(Y) \ar[l, "g^*"']
					\end{tikzcd} \end{center} are (up to homotopy equivalence) the vertical right-adjoints (resp. left-adjoints) of the squares   \begin{center}
					\begin{tikzcd}
						\D(X') \ar[r, "\bar{g}_!"]  & \D(X)\\
						\D(Y')\ar[u, "\bar{f}^*"]\ar[r, "g_!"] & \D(Y)\ar[u, "f^*"]
					\end{tikzcd}
				\ \ and \ \ 
					\begin{tikzcd}
						\D(X')   & \D(X) \ar[l, "\bar{g}^*"']\\
						\D(Y')\ar[u, "\bar{f}^*"] & \D(Y) \ar[u, "f^*"] \ar[l, "g^*"']
					\end{tikzcd} \end{center} along the canonical adjunctions.
\end{corollary}
\begin{proof}
    The same strategy as in the proof of Lemma \ref{lem:stably} applies; note that the left-to-right arrows in the diagrams (\ref{eq:square}) and (\ref{eq:bench}) can just as well be replaced by their $(-)^*$-counterparts. 
\end{proof}

\begin{proposition}\label{Different_definitions}
			Let $ \mathcal{D}: \te{Corr}(\mathcal{C}, E)^\otimes\to \te{Cat}_\infty $ be a six-functor formalism\footnote{With some care, one can reformulate \cite[Section 6]{scholze_notes} to make his definition apply to $ 3 $-functor formalisms. An according reformulation of our proof also applies in this generality.} and $ (f: X\to Y)\in E $. The following are equivalent:
				\begin{enumerate}
                    \item $f$ satisfies (Pr-II) (resp. (Ét-I)),
                    \item $f$ satisfies (Pr-II) (resp. (Ét-I)) stably,
				\item $ f $ is cohomologically proper (resp. cohomologically \'etale) in the sense of \cite[Definition 6.10 resp. Definition 6.12]{scholze_notes}.
				\end{enumerate}
		\end{proposition}

		\begin{proof}
            We treat the cohomologically proper case, the cohomologically \'etale case is dual.
            For $1\implies 3$ use induction and \cite[Proposition 6.11]{scholze_notes}. This reduces us to showing that the canonical morphism $f_! 1_X\simeq f_! p_{1, *} \Delta_{f, *} 1_X\simeq   f_! p_{1, *} \Delta_{f, !} 1_X\to f_*1_X $ coming by adjunction from
                \begin{align*}
                    f^*f_! p_{1, *} \Delta_{f, !} 1_X \simeq p_{2, !}p_1^*p_{1, *} \Delta_{f, !} 1_X \to p_{2, !}\Delta_{f, !} 1_X\simeq 1_X
                \end{align*}
            is an isomorphism.
            However, this is nothing but the isomorphism $f_!\to f_*$ provided by condition (Pr-II) in Definition \ref{defn:Pr_Et}, evaluated at $1_X$.

            For $3\implies 1$, use that, by induction, it suffices to show that 
                \begin{center}
                        \begin{tikzcd}
                            \mathcal{D}( X\times_Y X) \arrow[r, "{p_{2, !}}"]   & \mathcal{D}( X)                   \\
                            \mathcal{D}( X) \arrow[r, "f_!"] \arrow[u, "p_1^*"] & \mathcal{D}( Y) \arrow[u, "f^*"']
                        \end{tikzcd}
                \end{center}
            is vertically right-adjoinable. Cohomologically proper morphisms are stable under base change (since $f$-proper objects are, see also \cite[Remark 6.2]{scholze_notes}), so $p_1: X\times_Y X\to X$ is also cohomologically proper, and in particular, $p_1$ admits a right-adjoint.\footnote{If we have a six-functor formalism, this is tautological. But the corresponding statement for cohomologically \'etale morphisms is not. So we keep it for symmetry.}
            It remains to show that the adjunct $f_! p_{1, !}\to f_! p_{2, !}$ along the canonical adjunctions is an isomorphism. The fact that $1_X$ is $f$-proper provides maps
                \begin{align*}
                    \alpha: 1_Y\to f_! 1_X && \beta: 1_{X\times_YX}\to \Delta_{f,!}1_X
                \end{align*}
            such that
                \begin{align} \tag{$\star $}
 					1_X=f^*1_Y\overset{\alpha}{\to}f^*f_!1_X\simeq p_{2, !}p_1^*1_X\simeq p_{2, !}1_{X\times_Y X}\overset{\beta}{\to} p_{2, !}\Delta_{f, !} 1_X\simeq 1_X
 				\end{align}
            is the identity. Moreover, the canonical adjunctions $f^*\dashv f_!$ and $p_{1}^*\dashv p_{1, !}$ have unit 
                \begin{align*}
                    \te{id}\simeq \te{id}\otimes 1_Y\overset{\alpha}{\to} \te{id}\otimes f_! 1_X\simeq f_! f^*(\te{id}\otimes 1_Y)\simeq f_! f^*
                \end{align*}
            resp. counit
                \begin{align*}
                    p_1^*p_{1, !}&\simeq p_{13, !}(p_{12}^*\otimes p_{23}^*1_{X\times_Y X}) \overset{\beta}{\to} p_{13, !} (p_{12}^* \otimes p_{23}^*\Delta_{f,!}1_X) \\
                    &\simeq p_{13, !} (p_{12}^* \otimes (\te{id}, \Delta_{f})_!1_X) \simeq p_{13, !}(\te{id}, \Delta_f)_!(id, \Delta_f)^*p_{12}^*\simeq  \te{id}.
                \end{align*}
            We compare natural map $f_!p_{1,!}\to f_! p_{2,!}$ to  $f_!(p_{2, !}\otimes (\star))$. To this end define the natural isomorphism 
                \begin{align*}
                    \gamma_{A, B}: &f_! p_{2, !}p_{13, !}(p_{12}^*A\otimes p_{23}^*B)\simeq f_!p_{2, !}p_{12, !}(p_{12}^*A\otimes p_{23}^*B)\simeq f_!p_{2, !}(A\otimes p_{12, !}p_{23}^*B)\\
                    &\simeq f_! p_{2, !}(A\otimes p_{2}^* p_{1, !}B) \simeq f_!(p_{2, !}A\otimes p_{1, !}B).
                \end{align*}
            Consider the diagram
                \begin{center}
                        \begin{tikzcd}[row sep=small]
                            {(f_! p_{1, !})\otimes 1_Y} \arrow[r, "\simeq"] \arrow[d, "\alpha"]                                                              & {(f_!p_{2, !})\otimes 1_Y} \arrow[r, "\simeq"] \arrow[d, "\alpha"] \arrow[ld, "A", phantom] & {f_!(p_{2, !}\otimes f^* 1_Y)} \arrow[d, "\alpha"] \arrow[ld, "B", phantom]                 \\
                            {(f_!p_{1, !}) \otimes f_! 1_X} \arrow[r, "\simeq"] \arrow[d, "\simeq"]                                                          & {(f_!p_{2, !})\otimes f_!1_X}\arrow[r, "\simeq"]                                                               & {f_!(p_{2, !}\otimes f^*f_!1_X)} \arrow[ddd, "\simeq"] \arrow[llddd, "C", phantom]          \\
                            {f_!f^*f_!p_{1, !}} \arrow[d, "\te{BC}"]                                                                                         &                                                                                             &                                                                                             \\
                            {f_!p_{2, !}p_1^*p_{1, !}} \arrow[d, "\simeq"]                                                                                   &                                                                                             &                                                                                             \\
                            {f_!p_{2, !}p_{13, !}(p_{12}^*\otimes p_{23}^*1_{X\times_YX})} \arrow[rr, "{\gamma_{-,1_{X\times_Y X}}}"] \arrow[d, "\beta"]     &                                                                                             & {f_!(p_{2, !}\otimes p_{1, !}1_{X\times_Y X})} \arrow[d, "\beta"] \arrow[lld, "D", phantom] \\
                            {f_!p_{2, !}p_{13, !}(p_{12}^*\otimes p_{23}^*\Delta_{f, !}1_X)} \arrow[rr, "{\gamma_{-, \Delta_{f,!}1_X}}"] \arrow[d, "\simeq"] &                                                                                             & {f_!(p_{2, !}\otimes p_{1, !}\Delta_{f,!} 1_X)} \arrow[d, "\simeq"]                                   \\
                            {f_!p_{2, !}}                                                                                                                    &                                                                                             & {f_!p_{2,!}.}                                                                                                 
                    \end{tikzcd}
                \end{center}
            The squares $A$ and $B$ are filled by isomorphisms provided by the 6-functor formalism witnessing, that compositions of shriek functors resp. the projection formula are functorial and $D$ is filled with the functoriality constraint of $\gamma_{A, B}$. $C$ is filled as follows: We interpret it as a diagram in the anima of morphisms  $(X\times_Y X)\to (Y)$ in $\te{Corr}(\mathcal{C}^{\te{op},\sqcup, \te{op}}) $. Abstractly, all objects in the diagram are isomorphic to functor represented by the correspondence
                \begin{center}
                    \begin{tikzcd}
                        (X\times_YX)& (X\times_Y X\times_Y X)\arrow[l, "p_{12}"']\arrow[r]& (Y).
                    \end{tikzcd}
                \end{center}

         Carefully translating, the diagram $C$ is represented a diagram of isomorphisms between the apexes of the representing correspondences
            filled with the canonical isomorphisms coming from the universal properties of pullbacks.

            The composition of the right-most column is an isomorphism as $(\star)$ is and the composition of the left-most column is the right-adjoint of the base-change isomorphism. Hence the latter is also an isomorphism. This concludes $3\implies 1$.

            Since cohomologically proper morphisms are stable under base change, we get $3 \implies 2$. The implication $2\implies 1$ is trivial.
		\end{proof}

Lemma \ref{lem:stably} justifies the first half of the following definition, and Proposition \ref{Different_definitions} the second half.

\begin{definition}
	Given a two-functor formalism $\D$, we call a morphism $ f $ \emph{stably cohomologically proper}, if it is stably $ (\te{Pr-I}) $ or equivalently stably $ (\te{Pr-I}) $; and call it \emph{stably cohomologically proper} if it is stably $ (\te{\'Et- I}) $ or equivalently stably $ (\te{\'Et-II}) $.

    If $\D$ is a six-functor formalism, then we call a morphism \emph{cohomologically proper} if it is (stably) $ (\te{Pr-I}) $ or equivalently $ (\te{Pr-I}) $; and call it \emph{stably cohomologically proper} if it is (stably) $ (\te{\'Et- I}) $ or equivalently $ (\te{\'Et-II}) $.
\end{definition}

We arrive at the following definition.
\begin{definition}\label{defn:nagata}
    Given Nagata set-up $(\C,E,I,P)$, we call a two-functor formalism
    
    $\D:\te{Corr}(\C,E)\to \cat$
    \textit{Nagata} if every morphism in $P$ is stably cohomologically proper and every morphism in $I$ is stably cohomologically \'etale.
    We write $2FF(\C,E,I,P)$ for the full subcategory of $\Fun(\Corr(\C,E),\cat)$ spanned by two-functor formalisms that are Nagata. 
    
    We say that a three-functor formalism 
    $\D:\te{Corr}(\C,E)^\otimes \to \cat$ is Nagata if the corresponding two-functor formalism $ \te{Corr}(\mathcal{C}^{\te{op}, \sqcup, \te{op}}, E_\minus)\to \te{Cat}_\infty$ is Nagata w.r.t. the Nagata set-up $ (\mathcal{C}^{\te{op}, \sqcup, \te{op}}, E_\minus, I_\minus, P_\minus) $. We write $3FF(\C,E,I,P)$ for the full subcategory of 
    
    $\Fun(\Corr(\C,E)^\otimes,\cat)$ on three-functor formalisms that are Nagata. 
    
    We say that a six-functor formalism is Nagata if the underlying three-functor formalism is. We write $6FF(\C,E,I,P)$ for the full subcategory of $\Fun(\Corr(\C,E)^\otimes,\cat)$ on six-functor formalisms that are Nagata. 
\end{definition}



The following proposition characterizes natural transformations between Nagata two-functor formalisms.

 \begin{proposition}\label{prop:natural_transformation}
    Let $\D,{\D}':\Corr(\C,E)\to \cat$ be two-functor formalisms and $\alpha:\D\to \D'$ be a natural transformation. If $f:X\to Y$ in $E$ is cohomologically proper (resp. cohomologically \'etale), then the square 
    \begin{center}
        
            \begin{tikzcd}
                \D(X) \arrow[r, "\alpha_X"] \ar[d, "f_!"] & \D'(X) \ar[d,"f_!"]\\
                \D(Y) \ar[r, "\alpha_Y"] & \D'(Y)
            \end{tikzcd}
    
    \end{center} is (up to homotopy equivalence) the right (resp. left) adjoint of  \begin{center}
        
            \begin{tikzcd}
                \D(X) \arrow[r, "\alpha_X"]  & \D'(X) \\
                \D(Y) \ar[r, "\alpha_Y"] \ar[u, "f^*"] & \D'(Y) \ar[u, "f^*"]
            \end{tikzcd}
    
    \end{center} along the canonical adjunctions.
\end{proposition}
\begin{proof}
    This proof again follows the strategy of the proof of Lemma \ref{lem:stably}. For example, suppose that $f$ is cohomologically proper. Then we can apply the proof strategy to the  the diagram
\begin{center}
     \begin{tikzcd}
                            & \mathcal{D}(X) \arrow[rr, "\alpha_X"]                                                                              &                                    & \mathcal{D}'(X)                                                               \\
                            & \mathcal{D}(X\times_{Y} X) \arrow[rr, "\alpha_{X\times_Y X}"] \arrow[u, "\Delta_{f}^*"'] \arrow[ld, "{p_{1, !}}"'] &                                    & \mathcal{D}'(X\times_Y X) \arrow[u, "\Delta_{f}^*"'] \arrow[ld, "{p_{1,!}}"'] \\
                            \mathcal{D}(X) \arrow[rr, "\alpha_X"{xshift=-1.5em}]                    &                                                                                                                    & \mathcal{D}'(X)                    &                                                                               \\
                            & \mathcal{D}(X) \arrow[rr, "\alpha_X"{xshift=-1.5em}] \arrow[ld, "f_!"'] \arrow[uu, "p_{2}^*"'{yshift=1.5em}]                                    &                                    & \mathcal{D}'(X) \arrow[ld, "f_!"'] \arrow[uu, "p_{2}^*"']                     \\
                            \mathcal{D}(Y) \arrow[rr, "\alpha_Y"] \arrow[uu, "f^*"'] &                                                                                                                    & \mathcal{D}'(Y). \arrow[uu, "f^*"'{yshift=1.5em}] &                                                                              
                        \end{tikzcd}
\end{center}
    and take its vertical right-adjoint.
\end{proof}
Corollary \ref{cor:more_adjoinable_squares} and Proposition \ref{prop:natural_transformation} imply the following.
 \begin{corollary}\label{cor:same_as_nagata}
     For $(\C,E,I,P)$ a Nagata setup, the category $3FF(\C,E,I,P)$ coincides with the category $\mathbf{3FF}$ as defined in \cite{6ff}, and the category $6FF(\C,E,I,P)$ coincides with the category $\mathbf{6FF}$ as defined in loc. cit.
 \end{corollary}
 


\section{Scholze's conjecture}\label{sect:conjecture}
Let $(\mathcal{C}, E, I, P)$ be a Nagata set-up. In his thesis \cite[Theorem A.5.10]{mann_thesis}, Mann gives a convenient account of the Liu-Zheng construction, which requires the following properties of a functor
	\begin{align*}
		\mathcal{D}:\mathcal{C}^\op \to \textup{Cat}_\infty .
	\end{align*}


\begin{definition}\label{definition:BC functors}
 Let $(\mathcal{C},E,I,P)$ be a Nagata set-up.   Denote by $$\textup{BCFun}(\mathcal{C},E,I,P) \subseteq \textup{Fun}(\mathcal{C}^{\op},\textup{Cat}_\infty)$$ the subcategory spanned by functors
    $$\mathcal{D}:\mathcal{C}^\op \to \textup{Cat}_\infty $$ such that
    \begin{enumerate}
        \item for any cartesian diagram 	\begin{equation}\label{eq:another_generic_square}
								\begin{tikzcd}
									X' \ar[r, "\bar{g}"]\ar[d, "\bar{f}"]  & X\ar[d, "f"]\\
									Y'\ar[r, "g"] & Y,
								\end{tikzcd}
							\end{equation} if $f$ is in $P$, then then $f^*$ and $\bar{f}^*$ have right adjoints $f_*$ and $\bar{f}_*$ respectively, and 
       \begin{center}
									\begin{tikzcd}
										\mathcal{D}(X') & \mathcal{D}(X)\ar[l, "{\bar{g}^*}"'] \\
										\mathcal{D}(Y') \ar[u, "{\bar{f}^*}"']& \mathcal{D}(Y)\ar[l, "{g^*}"']\ar[u, "{f^*}"']
									\end{tikzcd}
								\end{center}
        is vertically right-adjoinable;
        \item for any cartesian diagram (\ref{eq:another_generic_square}), if $f$ is in $I$ then $f^*$ and $\bar{f}^*$ have left adjoints $f_\natural$ and $\bar{f}_\natural$ respectively, and 
	\begin{center}
									\begin{tikzcd}
										\mathcal{D}(X') & \mathcal{D}(X)\ar[l, "{\bar{g}^*}"'] \\
										\mathcal{D}(Y') \ar[u, "{\bar{f}^*}"']& \mathcal{D}(Y)\ar[l, "{g^*}"']\ar[u, "{f^*}"']
									\end{tikzcd}
								\end{center}
        is vertically left-adjoinable;
           \item for any cartesian diagram
							(\ref{eq:another_generic_square})
						with $ f\in I $ and $ g\in P $, the vertical left-adjoint
								\begin{center}
									\begin{tikzcd}
										\mathcal{D}(X')\ar[d, "{\bar{f}_\natural}"] & \mathcal{D}(X)\ar[l, "{\bar{g}^*}"']\ar[d, "{f_\natural}"] \\
										\mathcal{D}(Y') & \mathcal{D}(Y)\ar[l, "{g^*}"']
									\end{tikzcd}
								\end{center}				
						of the commutative diagram
								\begin{center}
									\begin{tikzcd}
										\mathcal{D}(X') & \mathcal{D}(X)\ar[l, "{\bar{g}^*}"'] \\
										\mathcal{D}(Y') \ar[u, "{\bar{f}^*}"']& \mathcal{D}(Y)\ar[l, "{g^*}"']\ar[u, "{f^*}"']
									\end{tikzcd}
								\end{center}
						is horizontally right-adjoinable. I.e. the canonical map $ f_\natural\bar{g}_*\to g_*\bar{f}_\natural $ is an isomorphism.
    \end{enumerate}
    and by natural transformations $\alpha:\mathcal{D}\to \mathcal{D'}$ such that for $f:X\to Y$ in $E$, the square
    \begin{center}
        \begin{tikzcd}
            \D(Y) \ar[d, "f^*"] \ar[r, "\alpha_Y"] & \D'(Y) \ar[d, "f^*"] \\
            \D(X) \ar[r, "\alpha_X"] & \D(X)
        \end{tikzcd}
    \end{center}
    is vertically right-adjoinable if $f$ is in $P$, and vertically left-adjoinable if $f$ is in $I$.

    Let $$\textup{BCFun}^{\textup{lax}}(\C,E,I,P) \subseteq \textup{BCFun}(\C^{\op, \sqcup,\op},E_-,I_-,P_-) $$
     be the full subcategory spanned by functors
     $$\C^{\op, \sqcup} \to \Cat_\infty $$
     that are lax cartesian structures.

     Let 
     $$\textup{BCFun}^{\textup{lax},L}(\C,E,I,P) \subseteq \textup{BCFun}^{\textup{lax}}(\C,E,I,P) $$
     denote the full subcategory spanned by lax cartesian structures 
          $$\mathcal{D}: \C^{\op, \sqcup} \to \Cat_\infty $$
          such that
  \begin{enumerate}
      \item[(a)] for every $X$ in $\C$, the symmetric monoidal category $\D(X)$ is closed,
      \item[(b)] for every $f:X\to Y$ in $\C$, $f^*$ has a right adjoint $f_*$,
      \item[(c)] for every $p:X\to Y$ in $P$, $f_*$ has a right adjoint $f^\natural$.
  \end{enumerate}
\end{definition}

\begin{theorem}[Scholze's conjecture]\label{theorem:scholzes conjecture}
    Let $(\mathcal{C},E,I,P)$ be a Nagata set-up.\linebreak Restriction along $ \mathcal{C}^{\te{op}}\hookrightarrow\te{Corr}(\mathcal{C},E) $ resp. $ \mathcal{C}^{\te{op}, \sqcup}\hookrightarrow\te{Corr}(\mathcal{C},E)^\otimes $ induces equivalences of categories 
    $$ \textup{BCFun}(\mathcal{C},E, I, P) \simeq \textup{2FF}(\mathcal{C},E,I,P),$$ $$ \textup{BCFun}^\lax(\mathcal{C},E, I, P)  \simeq \textup{3FF}(\mathcal{C},E,I,P) $$ and     $$ \textup{BCFun}^{\lax, L}(\mathcal{C},E, I, P)  \simeq \textup{6FF}(\mathcal{C},E,I,P). $$
\end{theorem}
\begin{proof}
We start by proving the first equivalence.
By applying \cite[Theorem 5.4]{LZb} twice, first with respect to the classes of morphisms $All^\op$ and $P^\op$, and then with respect to the classes of morphisms $All^\op$ and $I^\op$, we obtain categorical equivalences 
$$\delta_3^*\C(All^\op, I^\op,P^\op) \xrightarrow[\sim]{\phi_0} \delta_2^*\C(All^\op,I^\op) \xrightarrow[\sim]{\phi_1} \C^\op. $$
By condition 1, under the equivalence
$$ \phi^*:\Fun(\C^\op,\cat) \xrightarrow{\sim} \Fun(\delta_3^*\C(All^\op, I^\op,P^\op), \Cat_\infty) $$
induced by $\phi:= \phi_1\circ \phi_0$, the category
$\textup{BCFun}(\mathcal{C},E, I, P)$ is equivalent to a subcategory of $\Fun^{\{2\},R}(\delta_3^*\C(All^\op, I^\op,P^\op), \Cat_\infty)$ (see \cite[Section 2.4]{6ff} for the notation). By \cite[Theorem 2.3.2]{6ff}, there is an equivalence 
$${PA_{\{2\}}}: \Fun^{\{2\},R}(\delta_3^*\C(All^\op, I^\op,P^\op), \Cat_\infty) \xrightarrow{\sim} \Fun^{\{2\},L}(\delta_3^*\C(All^\op, I,P^\op), \Cat_\infty).$$
By conditions 2 and 3, under this equivalence $\psi^*(\textup{BCFun}(\mathcal{C},E, I, P))$ actually lands in a subcategory of 
$\Fun^{\{3\},L}(\delta_3^*\C(All^\op, I,P^\op), \Cat_\infty)$. 
After applying the equivalence
$${PA_{\{3\}}^{-1}} : \Fun^{\{3\},L}(\delta_3^*\C(All^\op, I,P^\op), \Cat_\infty) \xrightarrow{\sim}\Fun^{\{3\},R}(\delta_3^*\C(All^\op, I,P), \Cat_\infty),$$
we obtain a subcategory of 
$\Fun^{\{3\},R}(\delta_3^*\C(All^\op, I,P), \Cat_\infty)$. Lastly, by \cite[Theorem 5.4 and Theorem 4.27]{LZb}, there is an equivalence 
$$(\psi^*)^{-1}:\Fun(\delta_3^*\C(All^\op, I,P), \Cat_\infty) \xrightarrow{\sim} \Fun(\Corr(\C,E), \cat)$$
obtained by precomposing with the canonical categorical equivalence
$\psi^*:\delta_3^*\C(All^\op, I,P)$ $\to \Corr(\C,E)$.
 Keeping track of what happens to conditions 1-3 and the condition of morphisms when applying each of the equivalences above, we see that the image of \linebreak
 $\textup{BCFun}(\mathcal{C},E, I, P)$ in $\Fun(\Corr(\C,E), \cat)$ is in fact contained in $\textup{2FF}(\C,E,I,P)$, since by construction, for $\D$ in this image of $(\psi^*)^{-1}\circ PA_{\{3\}}^{-1}\circ PA_{\{2\}} \circ \phi^*$, morphisms in $P$ satisfy (Pr-I), and morphisms in $I$ satisfy (Ét-I). 

To show that the embedding
\begin{equation}\label{eq:conjecture}
 (\psi^*)^{-1}\circ PA_{\{3\}}^{-1}\circ PA_{\{2\}} \circ \phi^* : \textup{BCFun}(\mathcal{C},E, I, P) \longrightarrow \textup{2FF}(\C,E,I,P) 
\end{equation}
is in fact full and essentially surjective, and therefore an equivalence, we note that\linebreak 
$\psi^*(\textup{2FF}(\C,E,I,P))$ is a subcategory of $\Fun^{\{3\},R}(\delta_3^*\C(All^\op, I,P), \Cat_\infty),$ because of Proposition \ref{cor:more_adjoinable_squares} and Proposition \ref{prop:natural_transformation}. Moreover, its image after applying $PA_{\{3\}}$ is contained in  $\Fun^{\{2\},L}(\delta_3^*\C(All^\op, I,P^\op), \Cat_\infty)$, so we can apply $(\phi^*)^{-1}\circ PA_{\{2\}}^{-1}$. By construction, the image of $\textup{2FF}(\C,E,I,P)$ lands in $\textup{BCFun}(\mathcal{C},E, I, P)$. Therefore we get an embedding
$$(\phi^*)^{-1}\circ PA_{\{2\}}^{-1} \circ PA_{\{3\}} \circ \psi^*: \textup{2FF}(\C,E,I,P) \hookrightarrow  \textup{BCFun}(\mathcal{C},E, I, P) $$
it is the inverse of (\ref{eq:conjecture}) by definition. Lastly, we observe that restricting along $C^\op \hookrightarrow \Corr(\C,E)$ is clearly a left-inverse of $(\psi^*)^{-1}\circ PA^{-1}_{\{3\}} \circ PA_{\{2\}} \circ \phi^*$, and therefore equal to $(\phi^*)^{-1}\circ PA_{\{2\}}^{-1} \circ PA_{\{3\}} \circ \psi^*$.

For a Nagata setup of the form $(\C^{\op,\sqcup,\op},E_-,I_-,P_-)$, the functors $\phi^*$, $PA_{\{2\}}$, $PA_{\{3\}}^{-1}$ and $\psi^*$ and their inverses all preserve the existence of equivalences
$$\D((X_i)_I) \to \prod_I \D(X_i) $$
for $(X_i)_I$ in $\C^{\op,\sqcup, \op}$. Therefore (\ref{eq:conjecture}) restricts to an equivalence
$$ \textup{BCFun}^\lax(\mathcal{C},E, I, P)  \simeq \textup{3FF}(\mathcal{C},E,I,P).$$

Let $F$ be in $\textup{BCFun}^{\lax}(\mathcal{C},E, I, P)$, and let $\D$ be the corresponding three-functor formalism. First, suppose $F$ is in $\textup{BCFun}^{\lax,L}(\mathcal{C},E, I, P)$. $\D(X)$ remains closed monoidal for all $X$ in $\C$. For $i$ in $I$, there are adjunctions $i_\natural \vdash i^*\vdash i_*$, and for $p$ in $P$, there are adjunctions $p^* \vdash p_* \vdash p^\natural$. Moreover, since $\D$ is Nagata, there are adjunctions $i_!\vdash i^*$ and $p^* \vdash p_!$. Now for $f$ in $\C$ arbitrary, we can write $f=p\circ i$ for some $i$
 in $I$ and some $p$ in $P$. Then $f^* = i^* \circ p^*$ is the composite of left adjoints and hence a left adjoint. On the other hand, $f_! = i_! \circ f_!$ is the composite of left adjoints and hence a left adjoint. Therefore $\D$ is a Nagata six-functor formalisms.
 
 On the other hand, suppose $\D$ is a Nagata six-functor formalism. Then it is clear that $F$ satisfies (a) and (b). Moreover, for $f$ in $P$, $f^*$ has a right adjoint by (Pr-II). This shows that $F$ is in $\textup{BCFun}^{\lax,L}(\mathcal{C},E, I, P)$. Hence, (\ref{eq:conjecture}) restricts to an equivalence
 $$ \textup{BCFun}^{\lax,L}(\mathcal{C},E, I, P)  \simeq \textup{6FF}(\mathcal{C},E,I,P).$$
\end{proof}

\subsection{Comparison to coefficient systems}
\label{subsect:coeff}
For $\C$ the category $\textup{Sch}_B$ of finite type schemes over a finite-dimensional noetherian base scheme $B$, $E$ the separated morphisms, $P\subseteq E$ the proper morphisms and $I\subseteq E$ the étale morphisms, we can compare the category  $\textup{BCFun}^{\lax}(\C,E,I,P)$ to the category of coefficient systems in \cite[Definition 7.5]{drew_gallauer}. Coefficient systems are defined to encode $f^*$ and $f_*$ for all morphisms in $\C$, as well as a closed symmetric monoidal category $C(X)$ for all $X$ in $\C$. Base change and projection for smooth morphisms are imposed, as well as localization, $\mathbb{A}^1$-invariance and Tate-stability, and morphisms of coefficient systems are compatible with the left adjoint of $f^*$ for $f$ smooth. By \cite[Proposition 5.21]{drew}, morphisms between coefficient systems are also compatible with the right adjoint $f_*$ for $f$ proper. Moreover, coefficient systems satisfy proper base change, since this can be deduced from localization and smooth base change, see see \cite[Remark 3.4]{gallauer}. Lastly, condition 3 in Definition \ref{definition:BC functors}, also known as the support property, can be deduced from localization and proper base change, see \cite[Lemma 2.3.12]{Cisinski_deglise}.
    Therefore $\textup{CoSy}$ as defined in \cite[Definition 7.5]{drew_gallauer} is a full subcategory of  $\textup{BCFun}^{\lax}(\C,E,I,P)$. Let denote by $\textup{CoSy}^{L} \subseteq \textup{CoSy}$ the category of coefficient systems $C$ such that $C(X)$ is a cocomplete closed symmetric monoidal stable $\infty$-category for all $X$, and $f_*$ is a left adjoint for all $f$ in $P$, we call these \textit{left coefficient systems}.  Then $\textup{CoSy}^{L}$ is the subcategory of $\textup{BCFun}^{\lax,L}(\C,E,I,P)$ spanned by functors
    $\D:\C^\op \to \cat $ for which in addition
    \begin{itemize}
        \item for every $X$ in $\C$, $\D(X)$ is stable and cocomplete,
          \item for all smooth morphisms, $f^*$ has a left adjoint $f_\sharp$ satisfying base change and the projection formula,
          \item  conditions (3) and (4) in \cite[Definition 7.5]{drew_gallauer} hold,
    \end{itemize}
    and by natural transformations that are compatible with $f_\sharp$ for $f$ smooth.
    
Therefore Theorem \ref{theorem:scholzes conjecture} specializes to the following. 
\begin{theorem}
    Restriction gives an equivalence between  the category of coefficient systems $\textup{CoSy}^{L}$, and the subcategory spanned Nagata six-functor formalisms such that 
    \begin{itemize}
       \item for every $X$ in $\C$, $\D(X)$ is cocomplete,
          \item for all smooth morphisms, $f^*$ has a left adjoint $f_\sharp$ satisfying base change and the projection formula,
          \item  conditions (3) and (4) in \cite[Definition 7.5]{drew_gallauer} hold,
    \end{itemize}
    and natural transformations that are compatible with $f_\sharp$ for $f$ smooth.
\end{theorem}
Let us denote the full subcategory of $\textup{6FF}(\C,E,I,P)$ in the statement of the theorem above, by $\textup{6FF}(\C,E,I,P)^{DG}$. Note that $\textup{CoSy}^{L}$ is a full subcategory of $\textup{CoSy}^c$, and contains $\mathcal{SH}$. The theorem above, together with \cite[Theorem 7.14]{drew_gallauer}, implies the following.
\begin{corollary}
    The six-functor formalism $\mathcal{SH}$ is initial in  $\textup{6FF}(\C,E,I,P)^{DG}$.
\end{corollary}

Now, let us replace $\Sch_B$ in the definition of $\textup{CoSy}$ by the category $\Sch^{\textup{rf}}_B$ of reduced finite type schemes over a finite dimensional Noetherian base scheme, and separated morphisms between them. Then $(\Sch^{\textup{rf}}_B,I,P)$ is a Noetherian Nagata setup, see \cite[Definition A.1.12]{6ff}. Let us denote by $\textup{CoSy}^{\textup{Pr}} \subseteq \textup{CoSy}^L$ the category of left coefficient systems such $C(X)$ is presentable for all $X$, we call these presentable left coefficient systems. Because of condition (3) in \cite[Definition 7.5]{drew_gallauer}, $\textup{CoSy}^\textup{Pr}$ is in fact equivalent to a subcategory of $\mathbf{6FF}^\textup{loc}$. Therefore by \cite[Theorem 4.3.14]{6ff}, we get the following.
\begin{corollary}
   Restriction induces a faithful embedding of $\textup{CoSy}^\textup{Pr}$ into the category of lax symmetric monoidal functors
   $$(\textup{Prop}^{\textup{rf}}_B)^\op \to {\Pr}^L $$
   that satisfy descent for abstract blowups, where $\textup{Prop}^{\textup{rf}}_B  \subseteq \Sch^{\textup{rf}}_B$ is the subcategory of schemes that are proper over $B$. 
   
   In the case that $B = \textup{Spec}\ k$ for $k$ a field of characteristic zero,  restriction induces a faithful embedding of $\textup{CoSy}^\textup{Pr}$ into the category of lax symmetric monoidal functors
   $$(\textup{SmProp}^{\textup{rf}}_B)^\op \to {\Pr}^L $$
   that satisfy descent for abstract blowups, where $\textup{SmProp}^{\textup{rf}}_B  \subseteq \Sch^{\textup{rf}}_B$ is the subcategory of schemes that are smooth and proper over $k$.  
\end{corollary}

\subsection{Comparison to weaves}\label{subsect:khan}
    Let $\mathcal{D}:\Corr(\C,E)^\otimes \to \cat$ be a three-functor formalism. Then a morphism $f\in E$ satisfies (Pr-II) stably if and only if $f$ satisfies conditions (Pr1)-(Pr4) in \cite[Definition 2.18]{khan}; and $f$ satisfies (Ét-II) stably if and only if $f$ satisfies conditions (Sm1)-(Sm4) in \cite[Definition 2.30]{khan}. We can therefore compare our set-up to that in \cite{khan} in the following case. Suppose the category $\S$ in loc. cit. is sufficiently reasonable, in the sense that there is a Nagata set-up $(\S,E,I,P)$ with $E$ the set of morphisms that are in $\S'$, where the morphisms that we call ``proper'' form the class $P$, and the morphisms that we call ``étale'' are $I$. Suppose moreover that the morphisms that are called ``smooth'' in loc. cit. form a class $S\subseteq E$ that contains $I$.
    
    In this case, for a (left) preweave $\D:\Corr(\S,E) \to \cat$, to check that $\D$ admits $*$-direct images for proper morphisms (i.e., morphisms in $P$ satisfy conditions (Pr1)-(Pr4)), it suffices to check only conditions (Pr1) and (Pr4), since these imply that morphisms in $P$ satisfy (Pr-II) and hence also (Pr2) and (Pr3). Similarly, to check that $\D$ admits $\sharp$-direct images for morphisms in $I$, it suffices to check that morphisms in $I$ satisfy conditions (Sm1) and (Sm4) (since morphism in $I$ are closed under diagonal, condition (Sm5) is redundant, cf. \cite[Remark 2.34]{khan}). 
    
    For $\S$ as above, the category of left weaves (resp. weaves) is a subcategory of $3FF(\S,E,I,P)$ (resp. $6FF(\S,E,I,P)$), more precisely the subcategory of Nagata three-functor formalisms (resp. Nagata six-functor formalisms) such that for morphisms in $S$, the conditions (Sm1)-(Sm5) hold. 
    On the other hand, the category of $(*,\sharp,\otimes)$-formalisms (resp. adjoinable $(*,\sharp,\otimes)$-formalisms) is a subcategory of $\textup{BCFun}^\lax(\S,E,I,P)$ (resp. $\textup{BCFun}^{\lax,L}(\S,E,I,P)$), of one assumes that  $(*,\sharp,\otimes)$-formalism satisfy the support property, and that for an adjoinable $(*,\sharp,\otimes)$-formalism, $f_*$ is a left adjoint when $f$ is proper. Then the (adjoinable) $(*,\sharp,\otimes)$-formalisms are exactly the BC-functors such that not only for morphisms in $I$, but for all morphisms in the larger class $S$, conditions 2 and 3 in Definition \ref{definition:BC functors} hold, and such that morphisms in $S$ satisfy (Sm5). 
The equivalence in Theorem \ref{theorem:scholzes conjecture} then restricts to the following.

\begin{theorem}[{\cite[Theorem 2.51]{khan}}]
    For $\mathcal{S}$ as above, restriction induces an equivalence between the $\infty$-category of left weaves (resp. weaves) to the $\infty$-category of (resp. adjoinable) $(*,\sharp, \otimes)$-formalisms that satsify proper base change, projection and Thom stability.
\end{theorem}
The merit of this proof is that it uses only the machinery of Liu and Zheng and no $(\infty,2)$-categories; see also \cite[Warning 2.53]{khan}.

	\section{Relation to $ K $-theory}\label{sect:K-theory}
In Definition \ref{defn:nagata_setup}, we imposed the condition that for a Nagata set-up $(\C,E,I,P)$, all morphisms in $E$ need to be $n$-truncated for some $n$. In this section we exhibit an explicit counter-example, showing that this condition is necessary. This leads to the notion of \textit{twists of three-functor formalisms}, which are conjecturally classified by an invariant closely related to $K$-theory.
	
		\begin{example}[Twists]\label{Euler-characteristics}
			Let $ \te{FinAni}^\te{op}$ be the free $ \infty $-category with finite limits on a point, i.e., the opposite category of finite anima or equivalently of the homotopy category of finite CW complexes. Define
				\begin{align*}
					\chi: \te{Corr}(\te{FinAni}^\te{op})^\otimes\to \textbf{B}\mathbb{Z}^\otimes
				\end{align*}
			as the composition 
				\begin{align*}
					\te{Corr}(\te{FinAni}^\te{op})^\otimes\to \te{h}\te{Corr}(\te{FinAni}^\te{op})^\otimes\overset{\te{h}\chi}{\longrightarrow}\textbf{B} \mathbb{Z}^\otimes
				\end{align*}
			where $ \te{h}\chi $ is the strict\footnote{I.e. the maps $F(a)\otimes F(b)\to F(a\otimes b) $ and $1\to F(1)$ are identity morphisms.} symmetric monoidal $ 1 $-functor given explicitly as 
				\begin{enumerate}
					\item sending any object $ C $ of $ \mathcal{C} $ to the point $ \te{pt}$\footnote{The unique object in $ \textbf{B}\mathbb{Z} $.};
					\item sending an equivalence class of spans $ C\leftarrow E \to D $ (which is really a cospan $ C\to E\leftarrow D $ of finite CW complexes) to $ \chi(D)-\chi(E) $ the difference of Euler-characteristics of any (and thus all) representatives;
					\item this is a functor as it sends $ \te{id}_C\mapsto \chi(C)-\chi(C)=0=\te{id}_\te{pt}$ and a composition of spans corresponding to a composition of cospans of finite CW complexes 
						\begin{center}
							\begin{tikzcd}
								&                          & F\cup_D^\te{h} G                    &                          &                     \\
								& F \arrow[ru, "\bar{f}'"] &                                     & G \arrow[lu, "\bar{g}"'] &                     \\
								C \arrow[ru, "f"] &                          & D \arrow[lu, "g"'] \arrow[ru, "f'"] &                          & E \arrow[lu, "g'"']
							\end{tikzcd}
						\end{center}
					to $\chi(E)-\chi(F\cup_D^\te{h}G)= \chi(E)-(\chi(F)+\chi(G)-\chi(D))=(\chi(D)-\chi(F))+ (\chi(E)-\chi(G)) $;
					\item to see it is strong monoidal, it suffices to check that for cospans $A\to B\leftarrow C$ and $A'\to B'\leftarrow C'$,  $ \chi(C\sqcup C' )-\chi(B\sqcup B')=(\chi(C)-\chi(B))+(\chi(C')-\chi(B')) $.
				\end{enumerate}
			Let $ \mathcal{D}^\otimes $ be any symmetric monoidal $ \infty $-category with a strictly invertible object $ \mathcal{L} $ i.e. all braiding isomorphisms $ \mathcal{L}^{\otimes n } \to \mathcal{L}^{\otimes n} $ are identity maps.\footnote{For example, let $ R $ be an (underived) commutative ring and $ I $ be an invertible ideal. Then $ \mathcal{D}^\otimes=\te{Mod}_R $, the derived $ \infty $-category of $ R $-modules with $ \mathcal{L}=I[0] $ fits the bill.} Then $ \mathcal{L} $ induces a map $ \textbf{B}\mathbb{Z}^\otimes\to \te{Cat}_\infty$ sending $ \te{pt}\mapsto \mathcal{D}^\otimes $ and $ n\mapsto (-\otimes \mathcal{L}^{\otimes n}: \mathcal{D}\to \mathcal{D}) $. The composition 
				\begin{align*}
					\mathcal{D}_\mathcal{L}:\te{Corr}(\te{FinAni}^\te{op})^\otimes\to \textbf{B}\mathbb{Z}^\otimes\to \te{Cat}_\infty
				\end{align*}
			is a six-functor formalism whose restriction $ \mathcal{D}_\mathcal{L}^*: \te{FinAni}\to \te{CAlg}(\te{Cat}_\infty) $ factors over $ \{\te{pt}\}\overset{\mathcal{D}^\otimes}{\to} \te{CAlg}(\te{Cat}_\infty) $. The Liu-Zheng construction (regardless of what $ I $ and $ P $ you choose) attaches to this the constant six-functor formalism
				\begin{align*}
					\mathcal{D}: \te{Corr}(\te{FinAni}^\te{op})^\otimes\to \te{N}(Fin_*)\overset{\mathcal{D}^\otimes}{\to} \te{Cat}_\infty.
				\end{align*}
			 This is not equivalent to $ \mathcal{D}_\mathcal{L} $. But $ \mathcal{D}_\mathcal{L} $ satisfies all conditions of Definition \ref{defn:nagata}, besides that no morphism, that is not an equivalence, is $ n $-truncated in $ \te{FinAni}^\te{op} $ for any $ n\geq -1 $.
		\end{example}
		\begin{remark}
			The functor $ \chi: \te{Corr}(\te{FinAni}^\te{op})^\otimes\to \textbf{B}\mathbb{Z}^\otimes $ factors over the geometric realisation $ | \te{Corr}(\te{FinAni}^\te{op})^\otimes | $. It is noteworthy, that both restrictions $\chi^*: \te{FinAni}\to \textbf{B}\mathbb{Z} $ and $ \chi_!: \te{FinAni}^{\te{op}}\to \textbf{B}\mathbb{Z} $ are (isomorphic) to a constant map. For $ \chi^*$ it is literally true. There is a natural equivalence $ \chi_!\to c_\te{pt} $ to the constant functor with components $ \chi (M): \te{pt}=\chi_!(M)\to c_\te{pt}(M)=\te{pt} $.
			One can also see this from general principles. As they have an initial resp. terminal object, $ |\te{FinAni}|\simeq |\te{FinAni}^\te{op}|$ are contractible. But $ |\te{Corr}(\te{FinAni}^\te{op})| $ is not contractible!  

 Note that $ \delta_2^*\te{FinAni}^{\te{op}, \te{cart}}\to \te{Corr}(\te{FinAni}^\te{op}) $ is a categorical equivalence, in particular it induces a homotopy equivalence $  |\te{FinAni}^{\te{op}, \te{cart}}|= | \delta_2^*\te{FinAni}^{\te{op}, \te{cart}}\ |\simeq  | \te{Corr}(\te{FinAni}^\te{op})|$.
	
		\end{remark}
			
			 This leads to the following $ K $-theory-like invariant:
		\begin{definition}[$ T $-theory]
			Let $ \te{Cat}^\te{lex}_\infty $ be the $ \infty $-category of $ \infty $-categories with finite limits and left-exact functors between them. Define $ T $-theory
				\begin{align*}
					T: \te{Cat}^\te{lex}_\infty\to \te{Sp}_{\geq 0}
				\end{align*}
			to be the functor sending $ \mathcal{C}\mapsto \Omega|\te{Corr}(\mathcal{C})^\otimes| $, which is naturally a connective spectrum as $ \te{Corr}(\mathcal{C}) $ is connected. More precisely, the functor $ |\te{Corr}(-)^\otimes|: \te{Cat}^\te{lex}_\infty\to \te{CAlg}(\te{Ani}) $ factors uniquely as
                \begin{align*}
                    \te{Cat}^\te{lex}_\infty\overset{T}{\to}\te{Sp}_{\geq 0}\overset{\textbf{B}}{\to}\te{CAlg}(\te{Ani}).
                \end{align*}
		\end{definition}

        \begin{remark}
            This construction is well known in case $ \mathcal{C} $ is an exact $ \infty $-category. It is called the $ Q $-construction \cite{barwick_rognes} and there is a homotopy equivalence
                \begin{align*}
                    \Omega |\te{Corr}(\mathcal{C})|\simeq K(\mathcal{C})
                \end{align*}
            with algebraic $ K $-theory.
        \end{remark}
        
        \begin{remark} By \cite[Example 4.30.]{LZb} we have that $\delta_2^* \mathcal{C}^\te{cart}\to \te{Corr}(\mathcal{C})$ is a categorical equivalence, so, in particular a weak homotopy equivalence proving $ |\te{Corr}(\mathcal{C})|\simeq |\mathcal{C}^\te{cart}|$.
            The bisimplicial geometric realisation $ |\mathcal{C}^\te{cart}| $ can alternatively be computed as the colimit
                \begin{align*}
                    \varinjlim_{[n]\in \Delta^{\te{op}}} |\te{Fun}^{\te{cart}}(\Delta^n, \mathcal{C})^{ \otimes}|
                \end{align*}
            of the geometric realisations of the symmetric monoidal $\infty$-categories $ \te{Fun}^\te{cart}(\Delta^n, \mathcal{C})^{ \otimes}\subseteq \te{Fun}(\Delta^n, \mathcal{C})^\times $ whose maps are natural transformations consisting of cartesian squares.
        \end{remark}


        \begin{example}\label{bad_signs}
            We can already construct a (slightly less natural looking) $ 6 $-functor formalism, that satisfies the conditions of Theorem \ref{theorem:scholzes conjecture} and satisfies $ f^*\dashv f_! $ for all maps. Take $ \mathcal{C}=\mathcal{D}^\te{perf}(\mathbb{F}_3) $. It is well known, that $ K(\mathbb{F}_3)=\mathbb{Z}\times \textbf{B}U^{\te{h}\psi_3} $, where $\psi_3: \textbf{B}U\to \textbf{B}U$ is the third Adams operation. This induces a map\footnote{The map $\det: \textbf{B}U\to\textbf{B}U(1) $ is compatible with the third Adams operations, which on $\textbf{B}U(1)$ is given by $[L]\mapsto [L^3]$. $\textbf{B}U(1)^{\te{h}\psi_3}$ is thus computed by the fiber sequence $\textbf{B}\mathbb{Z}/2\to \textbf{B}U(1)\overset{[L]\mapsto [L^2]}{\longrightarrow} \textbf{B}U(1)$. }
                \begin{align*}
                    \te{Corr}(\mathcal{D}^\te{perf}(\mathbb{F}_3))^\otimes \to \textbf{B}K(\mathbb{F}_3)^\otimes\simeq \textbf{B}&(\mathbb{Z}\times \textbf{B}U^{\te{h} \psi_3})^\otimes\\
                     & \overset{pr_2}{\to} \textbf{B}^2(U^{\te{h}\psi_3})^\otimes\overset{\det}{\to} \textbf{B}^2 (U(1)^{\te{h}\psi_3})^\otimes \simeq \textbf{B}^2\mathbb{Z}/2^\otimes.
                \end{align*}
            Let $ \mathcal{D}^\otimes $ be a symmetric monoidal closed $ \infty $-category together with an isomorphism $ \phi: 1_\mathcal{D}\to 1_\mathcal{D} $ strictly of order $ 2 $. We obtain a six-functor formalism
                \begin{align*}
                    \mathcal{D}_\phi: \te{Corr}(\mathcal{D}^\te{perf}(\mathbb{F}_3))^\otimes \to \textbf{B}^2\mathbb{Z}/2^\otimes \to \te{Cat}_\infty.
                \end{align*}
        \end{example}
        \begin{remark}
            As $K(\mathbb{F}_3)$ is non-trivial in infinitely many degrees, even asking for base-change $n$-isomorphisms for arbitrary globally bounded $n$ to come from the adjunction $f^*\dashv f_!$ is insufficient to pin down the three-functor formalism. Asking it for all $n$ tautologically pins it down, but that's not much of a theorem.\footnote{In fact, \cite[Theorem 5.4]{LZb} directly implies, that a two-, three- or six-functor formalism is uniquely determined, if we fix \emph{all} data for maps in $ I $ and all data for maps in $ P $. This works even if only maps in $ I\cap P $ are $ n $-truncated for some $ n $, possibly depending on the map.}
        \end{remark}


        \begin{definition}\label{def:twist}
            Let $\mathcal{C}$ be an $\infty$-category with finite limits. Define the \emph{universal twist}
                \begin{align*}
                    T_{\mathcal{C}}: \mathcal{C}^\te{op}\to \te{CAlg}(\te{Cat}_\infty)
                \end{align*}
            as the composition
                \begin{align*}
                    \mathcal{C}^\te{op}\overset{\mathcal{C}_{-/}}{\to} \te{Cat}^\te{lex}\overset{T}{\to} \te{Sp}_{\geq 0}\hookrightarrow \te{CAlg}(\te{Ani})\hookrightarrow \te{CAlg}(\te{Cat}_\infty). 
                \end{align*}
            The first functor assigns $ X \mapsto \mathcal{C}_{X/}$, the coslice category under $X$ and maps $ (X\to Z) \mapsto (Y\to X\to Z) $; it is the straightening of the domain cartesian fibration $\textup{Arr}(\C) \to \C$, see also \cite[Corollary 2.4.7.12]{HTT}. 
            
            A \emph{twist of a $ 3 $-functor formalism} $ \mathcal{D}: \te{Corr}(\mathcal{C})^\otimes\to \te{Cat}_\infty $ is a symmetric monoidal natural transformation
                \begin{align*}
                    \mathcal{L}: T_\mathcal{C}\to \mathcal{D}^*, 
                \end{align*}
            where $\mathcal{D}^*$ corresponds to $ \mathcal{C}^{\te{op}, \sqcup}\to \te{Corr}(\mathcal{C})^\otimes\overset{\mathcal{D}}{\to }\te{Cat}_\infty $.
        \end{definition}

        \begin{remark}\label{explanation:twists}
            Given a three-functor formalism $\mathcal{D}: \te{Corr}(\mathcal{C})^\otimes\to \te{Cat}_\infty$, a map $\Omega |(\mathcal{C}_{X/})^\te{cart}| \to \mathcal{D}(X)$ is equivalent to a symmetric monoidal map
                \begin{align*}
                    \mathcal{L}: \delta_2^*(\mathcal{C}_{X/})^\te{cart}\to \te{B}\,\te{Pic} \mathcal{D}(X),
                \end{align*}
            where $\te{Pic}$ denotes the symmetric monoidal sub-anima of $\otimes$-invertible objects and isomorphisms between them.
            This sends a cartesian square $S$ under $X$ i.e. a $1$-simplex in $\delta_2^*(\mathcal{C}_{X/})^\te{cart}$
                \begin{center}
                        \begin{tikzcd}
                            X  \arrow[rd] \arrow[rdd, bend right, "f"'] \arrow[rrd, bend left] &\\
                            & A \arrow[d, "p"] \arrow[r] & B \arrow[d] \\
                            & C \arrow[r]           & D          
                        \end{tikzcd}
                \end{center}     
            to an object $\mathcal{L}_{X/S}$. As they are compatible with composition of cartesian squares, this only depends on the morphism $C\to D$ under $X$, so we may denote it by $\mathcal{L}_{X/C/D}$.
            Using the functoriality in $X$, $\mathcal{L}_{X/ C/ D}= f^* \mathcal{L}_{C/C/D}$, so we will drop $X$  from notation and just write $\mathcal{L}_{C/D}$ for the latter. 
            $2$-simplices encode base-change isomorphisms $\mathcal{L}_{A/ B}\simeq p^*\mathcal{L}_{C/ D}$ for base-change squares and $\mathcal{L}_{A/B}\otimes g^*\mathcal{L}_{B/C}\simeq \mathcal{L}_{A/C}$ for compositions $A\overset{g}{\to} B\overset{h}{\to} C$. Higher simplices encode higher coherences between these.

            The fact that $\mathcal{L}$ is symmetric monoidal encodes K\"unneth isomorphisms $\mathcal{L}_{A\times A'/ B\times B'}\simeq \mathcal{L}_{A/B}\boxtimes \mathcal{L}_{A'/B'}$. This \emph{should} be enough data to define a three-functor formalism $\mathcal{D}_\mathcal{L}: \te{Corr}(\mathcal{C})\to \te{Cat}_\infty$ with the same $f^*$'s, but for $f:A\to B$, $f_{!'}:= f_!(-\otimes \mathcal{L}_{A/B})$.

            From this data one also can deduce, that $\mathcal{L}_{C/D}\simeq \mathcal{L}_{C/1}\otimes f^*\mathcal{L}_{D/1}^{-1}$.
        \end{remark}
        
        \begin{lemma}\label{lem:QuillenQ}
            Let $\mathcal{C}$ be an $\infty$-category with finite limits and $X$ an object of $\mathcal{C}$. There is an isomorphism
                \begin{align*}
                    K(\te{Stab}(\mathcal{C}_{ X/, /X}))\simeq T_\mathcal{C}(X)
                \end{align*}
            of connective spectra.\footnote{$\mathcal{C}_{X/, /X}$ denotes the slice category of the coslice category $ \mathcal{C}_{X/} $ at $\te{id}_X: X\to X$.}
        \end{lemma}
        \begin{proof}
            We will construct an isomorphism of spectra between the Waldhausen S-construction for $ \mathcal{C}_{X/, /X} $ and the Quillen Q-construction for $ \mathcal{C}_{X/} $.
            
            This comparison will arise via a colax monoidal functor so we pass to opposite categories. Consider $\te{Fun}^{\te{cocart}}(\Delta^{n}, (\mathcal{C}_{X/})^\op)^\otimes$ $\subseteq \te{Fun}(\Delta^n, (\mathcal{C}_{X/})^\te{op})^\sqcup$ whose maps are natural transformations consisting of pushout squares. Further consider $\te{Fun}^\simeq(\Delta^{n-1}, (\mathcal{C}_{X/, /X})^\te{op})^\otimes\subseteq $ \linebreak
            $\te{Fun}(\Delta^{n-1}, (\mathcal{C}_{X/, /X})^\te{op})^\sqcup$ the maximal anima. 

            Together with the map $\te{N}(\te{Fin}_*)\to (\mathcal{C}_{X/})^{{\te{op}}, \sqcup}$ picking out $ (\te{id}_X: X\to X)\in (\mathcal{C}_{X/})^{{\te{op}}}\simeq $ \linebreak
            $\te{CAlg}((\mathcal{C}_{X/})^{{\te{op}}, \sqcup})$, they fit into a pullback square
            \begin{center}
                \begin{tikzcd}
                    \te{Fun}^\simeq(\Delta^{n-1}, (\mathcal{C}_{X/, /X})^\te{op})^\otimes \arrow[r]\arrow[d, hookrightarrow]& N(\te{Fin}_*)\arrow[d, hookrightarrow]\\
                    \te{Fun}^{\te{cocart}}(\Delta^{n}, (\mathcal{C}_{X/})^\op)^\otimes\arrow[r, "\te{ev}_{[0]}"]& (\mathcal{C}_{X/})^{\te{op}, \sqcup}
                \end{tikzcd}
            \end{center}
                
        of simplicial sets over $ \te{N}(\te{Fin}_*) $. This also implies, that $\te{Fun}^\simeq(\Delta^{n-1}, (\mathcal{C}_{X/, /X})^\te{op})^\otimes\to$ \linebreak
        $\te{Fun}^{\te{cocart}}(\Delta^{n}, (\mathcal{C}_{X/})^\op)^\otimes$ is lax symmetric monoidal. In particular, we get a map
            \begin{align*}
                |\te{Fun}^\simeq(\Delta^{n-1}, (\mathcal{C}_{ X/, /X})^\te{op})^\otimes|\to |\te{Fun}^{\te{cocart}}(\Delta^{n}, (\mathcal{C}_{X/})^\op)^\otimes|
            \end{align*}
        in $\te{CAlg}(\te{Ani})$. To check it is an isomorphism, we appeal to Quillen's Theorem A. Computing the comma-categories at objects in the target is simple as the functor is a fully-faithful embedding. Given $A_\bullet=(A_0\leftarrow A_1\leftarrow \cdots\leftarrow A_n) \in \te{Fun}^{\te{cocart}}(\Delta^{n}, (\mathcal{C}_{X/})^\op)$, we have to check that the full subcategory of $\te{Fun}^{\te{cocart}}(\Delta^{n}, (\mathcal{C}_{X/})^\op)_{A_\bullet\backslash}$ on $A_\bullet\to B_\bullet $ with $B_\bullet\in \te{Fun}^\simeq(\Delta^{n-1}, (\mathcal{C}_{ X/, /X})^\te{op})$ has contractible geometric realisation. It is already a contractible anima as it just the anima of left Kan-extensions of the diagram $\Delta^n\sqcup_{\te{ev}_{[0]}, \Delta^0, \te{ev}_{[0]}} \Delta^1\to \mathcal{C}_{X/}$ given by
            \begin{center}
                \begin{tikzcd}
                    A_0& \arrow[l] A_1 & \arrow[l] \cdots & A_n\arrow[l]\\
                    X\arrow[u]
                \end{tikzcd}
            \end{center}
         along $\Delta^n\sqcup_{\te{ev}_{[0]}, \Delta^0, \te{ev}_{[0]}} \Delta^1\hookrightarrow \Delta^n\times \Delta^1$. On the level of anima, the passage to opposite categories\footnote{Note the reindexing. To obtain $\te{Fun}(\Delta^n, (\mathcal{C}_{X/}))^{\te{op}}\simeq \te{Fun}(\Delta^n, (\mathcal{C}_{X/})^{\te{op}} )$ we need to use the isomorphism $\Delta^n\simeq \Delta^{n, \te{op}}$.} has no effect, so we obtain an isomorphism
            \begin{align*}
                |\te{Fun}^\simeq(\Delta^{n-1}, \mathcal{C}_{X/, /X})^\otimes|\simeq |\te{Fun}^{\te{cart}}(\Delta^{n}, \mathcal{C}_{X \backslash})^\otimes|.
            \end{align*}

        This is natural in $[n]\in \Delta^{\te{op}}$ yielding an equivalence
                \begin{align*}
                    K(\mathcal{C}_{X/,/X})=\Omega\varprojlim_{n\in \Delta^{op}} \vert\te{Fun}^\simeq (\Delta^{n-1}, \mathcal{C}_{X/,/X})^{ \otimes}\vert \simeq \Omega\varprojlim_{n\in \Delta^{op}} \vert\te{Fun}^{\te{cart}}(\Delta^n, \mathcal{C}_{X/})^{ \otimes}\vert = T_\mathcal{C}(X)
                \end{align*}
            of spectra. 

            For pointed $\infty$-categories admitting finite limits like $ \mathcal{C}_{X/, /X} $, combine \cite[Definition 2.16, Proposition 8.3]{Barwick} to deduce
            $K(\mathcal{C}_{X/,/X})\simeq K(\te{Stab}(\mathcal{C}_{X/,/X}))$.
        \end{proof}
    
    
    	\begin{remark}\label{def:generaltwists}
    		Using this reformulation, we can also define twists for general geometric set-ups $ (\mathcal{C}, E) $. For any object $ X $, consider $ \mathcal{C}_{E, X/, /X} := ((\mathcal{C}_E)_{/X})_{{(\te{id}_X: X\to X)}/}$, where $ \mathcal{C}_E $ is the subcategory of $ \mathcal{C} $ spanned by morphisms in $E$. For any morphism $ X\to Y $, pulling back yields a map $\mathcal{C}_{E, Y/, /Y}\to \mathcal{C}_{E, X/, /X}$ preserving finite limits. As before, this assembles into a functor 
    			\begin{align*}
    				\mathcal{C}_{E, -/,/-}: \mathcal{C}^{\te{op}}\to \te{Cat}_\infty^\te{lex},
    			\end{align*}
    		composing it with the functor $ \Omega|\te{Corr}(-)| : \te{Cat}_\infty^\te{lex}\to \te{Sp}_{\geq 0}\to \te{CAlg}(\te{Cat}_\infty)$ of Definition \ref{def:twist}, we get 
    			\begin{align*}
    				T_{(\mathcal{C}, E)}:  \mathcal{C}^\te{op}\to \te{Cat}_\infty.
    			\end{align*}
    		A \emph{twist} of a three-functor formalism $ \mathcal{D}: \te{Corr}(\mathcal{C}, E)^\otimes\to \te{Cat}_\infty $ is accordingly defined as a symmetric monoidal natural transformation $ T_{(\mathcal{C}, E)}\to \mathcal{D}^* $.
    		Combining Remark \ref{explanation:twists} and Lemma \ref{lem:QuillenQ}, twists should again correspond to three functor formalisms with the same $ \otimes $'s and $ f^* $'s but $ f_! $'s replaced by $ f_{!'}= f_!(-\otimes \mathcal{L}_f) $, where, given $ f: X\to Y $, $ \mathcal{L}_{f} $ is the image of $ (X\to \Omega^2_X Y\to X) $ under $ |\te{Corr}(\mathcal{C}_{E, X/, /X})|\to \te{Pic}\mathcal{D}(X) $; and all coherences tweaked in a similar fashion as explained in Remark \ref{explanation:twists}.\footnote{Instead of $\Omega^2_X$, we could have chosen any $ \Omega^{2n}_{X} $ for $ n\geq 1 $. We need to pass to even powers to be compatible with Definition \ref{def:twist} because the isomorphism of Lemma \ref{lem:QuillenQ} multiplies by $(-1)$.}
    	\end{remark}
        
        \begin{proposition}\label{prop:twists on affine schemes}
            Let $\mathcal{C}$ be the opposite category of $\mathbb{E}_\infty$-rings over some $\mathbb{E}_\infty$-ring $R$ \footnote{As this base does not really matter, the same statement applies when working over any presheaf of anima on the $\infty$-category of $\mathbb{E}_\infty$-rings.} i.e. the category of derived affine schemes over $\te{Spec}(R)$ and $E$ the class of maps of finite presentation. Let $S$ be an $R$-algebra, then
                \begin{align*}
                    T_{(\mathcal{C}, E)}(S)= K(S),
                \end{align*}
            where $K$ is the connective algebraic K-theory of an $\mathbb{E}_\infty$-ring.
            \end{proposition}
            \begin{proof}
            First note that $\te{Stab}(\te{CAlg}^{\te{f.p.}, \te{op}}_{S/, /S})= \te{Stab}(\te{CAlg}^{\te{f.p.}}(S)^\te{op})$. We will identify the latter with $\te{Mod}_S^\te{perf}$ in a round-about way.
            In \cite[Corollary 7.3.4.14]{HA}, Lurie computes $\te{Stab}(\te{CAlg}(S)_*)=\te{Mod}_S$, which by e.g. \cite[Remark 4.2]{harpaz} is equivalent to
            $\te{Ind}(\te{Stab}(\te{CAlg}^{\te{f.p.}}(S)_*^\te{op})^{\te{op}})$. Filtered colimits and opposite categories of idempotent complete $\infty$-categories remain idempotent complete. Thus the fact that $\te{CAlg}(S)_*^\te{f.p.}$ is idempotent complete implies the same holds for $\te{Stab}(\te{CAlg}^{\te{f.p.}}(S)_*^\te{op})^{\te{op}}$. By \cite[Proposition 7.2.4.2]{HA}, the $\infty$-category of compact objects of $\te{Mod}_S$ is $\te{Mod}_S^\te{perf}$. This implies $\te{Stab}(\te{CAlg}^{\te{f.p.}}(S)_*^\te{op})^{\te{op}}\simeq \te{Mod}_S^\te{perf}$. Furthermore, \cite[Proposition 7.2.4.4]{HA} states that passing to duals induces an equivalence of stable $\infty$-categories $\te{Mod}_S^{\te{perf},\te{op}}\overset{\simeq}{\to}\te{Mod}_S^{\te{perf}}$. Finally, we arrive at $T_\mathcal{C}(S)\simeq $ $K(\te{Mod}_S^\te{perf})=K(S)$.
            \end{proof}
        
        	\begin{example}\label{great_example}
        		Let $ \mathcal{C} $ be as in \ref{prop:twists on affine schemes} and $\mathcal{D}:  \te{Corr}(\mathcal{C}, E)^\otimes\to \te{Cat}_\infty$ a three-functor formalism such that $ \mathcal{D}^* $ sends $ X\to \te{QCoh}(X) $, $X\to \te{IndCoh}(X)$ or some variation thereof. Consider the symmetric monoidal natural transformation
        			\begin{align*}
        				T_{(\mathcal{C}, E)} \to \mathcal{D}^* 
        			\end{align*}
        		given by the determinant map $ K(R)\overset{\det}{\to} \te{Pic}(R)\hookrightarrow \mathcal{D}^*(\te{Spec}(R)) $. 
        		
        		This corresponds to the twisted three-functor formalism with lower-shrieks 
        			\begin{align*}
        				f_{!'}= f_!(-\otimes \det \mathbb{L}_f),
        			\end{align*}
				 where $ \mathbb{L}_f $ is the relative cotangent complex.\footnote{The relative cotangent complex appears through the identification $ \te{Stab}(\te{CAlg}^{\te{f.p.}}(R))\simeq \te{Mod}^\te{perf}_R $.}
        	\end{example}

    In the spirit of the first part of the paper, we have:
		\begin{theorem}
			Let $ \mathcal{C} $ be an $ \infty $-category with finite limits, such that all morphisms of $ E $ are $ n $-truncated for some $ n $, then $ T_{(\mathcal{C}, E)}=0 $.
		\end{theorem}
        \begin{proof}
            For any object $ X\in \mathcal{C} $, $ \mathcal{C}_{E,  X/, /X} $ still has the property, that every morphism is $ n $-trunca-
            ted. In particular, for any object $Y$, $0\to Y $ is $n$-truncated i.e. $\Omega^{n-1}(Y)=0$.
            Thus, as $\Omega_X$ induces an isomorphism on
            $ |(\mathcal{C}_{E, X/, /X})^\te{cart}|$ and $\varprojlim_n \Omega^n_X=0$, it must be $0$.
        \end{proof}

    \section{Outlook}\label{sect:outlook}

    In this section we collect conjectural generalisations of Scholze's conjecture. 
    
    \begin{remark}
    	In Remark \ref{explanation:twists}, we explained how a twist $ \mathcal{L}: T_\mathcal{C}\to \mathcal{D}^* $ should provide us with another three-functor formalism $ \mathcal{D}_\mathcal{L} $. This other three-functor formalism is defined purely in terms of the cohomological correspondences $f_!(-\otimes \mathcal{L}_f)$. So it naturally lifts to a symmetric monoidal functor $\mathcal{D}_\mathcal{L}: \te{Corr}(\mathcal{C})^\otimes\to \mathfrak{LZ}^\otimes_\mathcal{D} $ into the symmetric monoidal $ (\infty, 2) $-category of cohomological correspondences, see e.g. \cite[Section 2.2]{Zav23}. As $ \mathcal{D}^* $ remains unchanged, this yields a commutative diagram
    		\begin{center}
    			\begin{tikzcd}
    						& \mathcal{C}^{\te{op}, \sqcup}\ar[dl, hookrightarrow]\ar[dr, "\mathcal{D}^*"]\\
    				\te{Corr}(\mathcal{C})^\otimes\arrow[rr, "\mathcal{D}_\mathcal{L}"] && \mathfrak{LZ}_\mathcal{D}^\otimes.
    			\end{tikzcd}
    		\end{center}
    	
    	Any such diagram yields a three-functor formalisms via $ \te{Corr}(\mathcal{C})^\otimes\to \mathfrak{LZ}_\mathcal{D}^\otimes \to \te{Cat}_\infty$, which leaves $ \mathcal{D}^* $ unchanged. On the other hand, the correspondence $1 \leftarrow \Delta_X\hookrightarrow X\times X $ induces two self-duality data on $ X $ in $ \mathfrak{LZ}_\mathcal{D} $ via $ \mathcal{D} $ resp. $ \mathcal{D}_\mathcal{L} $ such that $f_!$ resp. $f_!(-\otimes \mathcal{L}_f)$ are identified with the dual map of $ f^* $. Any two duality data differ by an isomorphism, so we get an isomorphism of $ X $ in $ \mathfrak{LZ}_\mathcal{D} $ i.e. a $ \otimes $-invertible object $ \mathcal{L}_{X/1} \in \mathcal{D}(X)$. The data identifying the duals of $ f^* $ w.r.t. these two duality data should yield all coherence conditions to fashion a twist $ T_\mathcal{C}\to \mathcal{D}^* $.

    	This relates directly to the main question of this article as both of the properties ``cohomologically proper" and ``cohomologically \'etale" can intrinsically  be formulated in the symmetric monoidal $ (\infty, 2) $-category $ \mathfrak{LZ}^\otimes_\mathcal{D} $ of cohomological correspondences as soon as we pick out the $f^*$'s by the symmetric monoidal functor $ \mathcal{C}^{\te{op}, \sqcup}\to \mathfrak{LZ}_\mathcal{D}^\otimes $.

    \end{remark}
    
    \begin{conjecture}\label{conjecture:LZ}
 			Let $\mathcal{D}: \te{Corr}(\mathcal{C})^\otimes\to \te{Cat}_\infty$ a three-functor formalism. Denote by $\mathfrak{LZ}_\mathcal{D}^\otimes$ its symmetric monoidal $(\infty, 2)$-category of cohomological correspondences with underlying symmetric monoidal $\infty$-category $\te{LZ}_\mathcal{D}^\otimes$. There is a canonical isomorphism of anima
 			\begin{align*}
 				\te{Hom}_{\te{CAlg}(\te{Cat}_{\infty})_{ \mathcal{C}^{\te{op}, \sqcup}/}}(\te{Corr}(\mathcal{C})^\otimes, \textnormal{LZ}_\mathcal{D}^\otimes)\simeq
     \te{Hom}_{\te{Fun}(\mathcal{C}^\te{op}, \te{CAlg}(\te{Cat}_\infty))}(T_\mathcal{C}, \mathcal{D}^*).
 			\end{align*}
 	\end{conjecture}
    
    \begin{conjecture}
        Let $\mathcal{C}$ be an $\infty$-category with finite limits. For any natural number $n$, there exists a categorical equivalence
            \begin{align*}
                \delta_{n,1}^*\mathcal{C}^{\te{cart}, \otimes}\to \underbrace{\te{Corr}(\mathcal{C})^\otimes\otimes_{\mathcal{C}^{\te
                op}, \sqcup}\cdots \otimes_{\mathcal{C}^{\te
                op}, \sqcup} \te{Corr}(\mathcal{C})^\otimes}_{n\te{ times}}
            \end{align*}
        with $\otimes$ denoting the homotopy pushout in $\te{CAlg}(\te{Cat}_\infty)$.
    \end{conjecture}

    \begin{remark}
        A symmetric monoidal functor $F: \te{Corr}(\mathcal{C})^\otimes\otimes_{\mathcal{C}^{\te
                op}, \sqcup}\cdots \otimes_{\mathcal{C}^{\te
                op}, \sqcup} \te{Corr}(\mathcal{C})^\otimes\to \mathcal{E}^\otimes$
        is nothing but $n$ symmetric monoidal functors $F_i: \te{Corr}(\mathcal{C})^\otimes\to \mathcal{E}^\otimes$ that agree on $\mathcal{C}^{\te{op}, \sqcup} $. For example, take $\mathcal{E}^\otimes=\te{LZ}_\mathcal{D}^\otimes$ and $F_1: \te{Corr}(\mathcal{C})^\otimes \to \te{LZ}_\mathcal{D}^\otimes$ the canonical functor. Then according to Conjecture \ref{conjecture:LZ}, $F_i$ are given by twists $\mathcal{L}_i: T\to \mathcal{D}^*$ and thus, for any pullback square
            \begin{center}
                \begin{tikzcd}
                    W\ar[r, "\bar{g}"]\ar[d, "\bar{f}"] & Z\ar[d, "f"] \\
                    X\ar[r, "g"] & Y
                \end{tikzcd}
            \end{center}
        we get an isomorphism $ F_i(f_!)\circ F_j(\bar{g}_!)\simeq f_!(\bar{g}_!(-\otimes \mathcal{L}_{j, \bar{g}})\otimes \mathcal{L}_{i, f})\simeq f_!\bar{g}_!(-\otimes \mathcal{L}_{j, \bar{g}}\otimes \bar{g}^*\mathcal{L}_{i, f})\simeq g_!\bar{f}_!(-\otimes \bar{f}^*\mathcal{L}_{j, g}\otimes \mathcal{L}_{i, \bar{f}})
        \simeq F_j(g_!)\circ F_i(\bar{f}_!) $. Together, this assembles into a map $\hat{F}:\delta^*_{n,1} \mathcal{C}^{\te{cart}, \otimes}\to \te{LZ}_\mathcal{D}^\otimes$. Conversely, given such a map, we define the twists $\mathcal{L}_i$ in terms of
         $\delta_2^*(\mathcal{C}_{X/, /X}^\te{cart} )\to \te{BPic}(\mathcal{D}(X))$ by sending a pullback square
            \begin{center}
                    \begin{tikzcd}
                        X \arrow[r, "a"] & A \arrow[r, "\bar{g}"] \arrow[d, "\bar{f}"'] & B \arrow[d, "f"] &   \\
                         & C \arrow[r, "g"']                            & D \arrow[r, "d"] & X
                    \end{tikzcd}
            \end{center} 
        to $d_! \phi a_!1_X$, where $\phi: \hat{F}(A)\to \hat{F}(D)$ is $\hat{F}$ applied to the square in the $(1, j)$-direction.

    \end{remark}

\bibliographystyle{alphaurl}		
\bibliography{bibliography}    
\end{document}